\documentclass[reqno, 11pt]{amsart}

\usepackage[utf8]{inputenc}
\usepackage{hyperref}
\usepackage{enumitem}
\usepackage[margin=2cm]{geometry}
\usepackage{amsfonts,amsmath,amsthm,amssymb}
\usepackage{graphicx}
\usepackage{siunitx}
\usepackage{mathrsfs} 
\usepackage{xcolor}

\numberwithin{equation}{section}
\newtheorem{theorem}{Theorem}[section]
\newtheorem{corollary}[theorem]{Corollary}
\newtheorem{lemma}[theorem]{Lemma}
\newtheorem{proposition}[theorem]{Proposition}
\theoremstyle{definition}

\theoremstyle{remark}
\newtheorem{remark}[theorem]{Remark}

\DeclareMathOperator{\ep}{\varepsilon}

\DeclareMathOperator{\Img}{Im}

\DeclareMathOperator{\R}{\mathbb{R}}
\DeclareMathOperator{\C}{\mathbb{C}}
\DeclareMathOperator{\T}{\mathbb{T}}
\DeclareMathOperator{\D}{\mathbb{D}}

\allowdisplaybreaks 

\title[]{Global bifurcation for corotating and counter-rotating vortex pairs}

\author{Claudia Garc\'ia}
\address{Departamento de Matem\'aticas, Universidad Aut\'onoma de Madrid, Ciudad Universitaria de Cantoblanco, 28049, Madrid, Spain  \& Research Unit ``Modeling Nature'' (MNat), Universidad de Granada, 18071 Granada, Spain}
\email{claudia.garcial@uam.es}

\author{Susanna V.~Haziot}
\address{Department of Mathematics, Brown University, Box 1917, Providence, RI 02912, USA}
\email{susanna\_haziot@brown.edu}

{\thanks{C.G has been supported by the European Research Council through Grant ERC-StG-852741 (CAPA), the MINECO--Feder (Spain) research grant number RTI2018--098850--B--I00 and the Junta de Andaluc\'ia (Spain) Project
FQM 954, S.V.H is partially supported by the National Science Foundation through the award DMS-2102961. This work was developed during the semester program {\it Hamiltonian Methods in Dispersive and Wave Evolution Equations} at ICERM}

\begin{document}

\date{\today}

\begin{abstract}
The existence of a local curve of corotating and counter-rotating vortex pairs was proven by Hmidi and Mateu in \cite{hm:pairs} via a desingularization of a pair of point vortices. In this paper, we construct a global continuation of these local curves. That is, we consider solutions which are more than a mere perturbation of a trivial solution. Indeed, while the local analysis relies on the study of the linear equation at the trivial solution, the global analysis requires on a deeper understanding of topological properties of the nonlinear problem. For our proof, we adapt the powerful analytic global bifurcation theorem due to Buffoni and Toland, to allow for the singularity at the bifurcation point. For both the corotating and the counter-rotating pairs, along the global curve of solutions either the angular fluid velocity vanishes or the two patches self-intersect.  
\end{abstract}

\maketitle

\setcounter{tocdepth}{1}

\tableofcontents

\section{Introduction}
We will consider rotating vortex patch solutions to the two-dimensional Euler equations in vorticity form. A vast amount of work has been done for the dynamical solutions to this nonlinear transport equation, such as local and global well-posedness of classical solutions in various function spaces. However, we are interested in \emph{steady} solutions, that is, solutions where the patches are either rotating at constant angular velocity, or translating at constant speed, without changing shape. These are also sometimes referred to as \emph{relative equilibria} or \emph{V-states}. There are a number of known, explicit solutions to this problem, such as the disk, an annulus, Kirchhoff ellipses, \cite{kirchhoff:book}. Perturbative existence results bifurcating from these then yield more complex, polygonal-like solutions.

We will focus on the situation when the explicit solutions are rotating or translating configurations of two point vortices. They can be perceived as limiting, singular cases of a very small vortex patch. Indeed, when perturbatively desingularized, one obtains pairs of symmetric patches with either same (\emph{corotating vortex pairs}) or opposite (\emph{counter-rotating vortex pairs}) circulations. Corotating patches rotate about the center of the system with constant angular velocity \cite{hm:pairs} and counter-rotating ones translate steadily at constant speed. Although this paper only focuses on the scenario of two point vortices, work has been done for an arbitrary number $N$ of points \cite{garcia:karman, garcia:choreography, Hassainia-Wheeler-points}.

Most of the work done on multiple and isolated V-state vortex patches is of perturbative nature. Indeed, the solutions obtained are very close to the explicit solutions they are bifurcating from, and are constructed by means of either the celebrated local bifurcation theorem of Crandall and Rabinowitz, or by the implicit function theorem. However, numerics suggest that interesting and beautiful solutions form as one moves further away from the trivial solutions, such as ones containing cusps or corners. Although there are to date no analytic results proving the formation of corners, in order to even get closer to understanding such solutions, the local curve of solutions needs to be extended to a global one. This has been done by Hassainia, Masmoudi, and Wheeler in \cite{hmw:global} for the case of the simple vortex patch bifurcating from the disk. By means of the global analytic bifurcation theorem due to Dancer \cite{dancer:global} and Buffoni and Toland \cite{bt:analytic}, they obtain a global curve which limits to a vanishing of the angular fluid velocity. To the best of our knowledge, this is the only global bifurcation result for V-state vortex patches.

The aim of this paper is to construct a global curve of solutions for the corotating and counter-rotating vortex pairs. As in \cite{hmw:global}, we will do this by using an analytic global bifurcation theorem. However, our case is more complicated than the one studied in \cite{hmw:global}. Indeed, the presence of more than one patch does not only add an additional layer of complexity to the problem but it also yields an extra limiting alternative along the curve: intersection of the boundary of the two patches. Moreover, since the explicit solutions to the problem are points rather than patches, the formulation of the problem contains a singularity which requires us to adapt the global bifurcation theorem accordingly. We first carry out the entire analysis for the corotating pairs and in the final section, show how the result can easily be adapted to that of the counter-rotating ones.      
\subsection{Presentation of the problem for corotating vortex pairs}
We begin by recalling the two-dimensional incompressible Euler equations expressed in the vorticity form 
\begin{equation}\label{euler}
\partial_t\omega+(u\cdot\nabla)\omega=0,\qquad u=\nabla^\perp\psi, \qquad\Delta\psi=\omega.
\end{equation}
Here $u$ denotes the velocity field, $\psi$ the stream function and $\omega$ the vorticity. For the remainder of the paper, we identify $(x,y)\in\mathbb{R}^2$ with $z=x+iy\in\mathbb{C}$ and we denote $z^\perp=(-y,x)$.

We now seek weak solution of \eqref{euler} which satisfy the initial data 
\begin{equation}\label{initial_vorticity}
\omega_0(z):=\omega(0,z)=\frac{1}{\ep^2\pi}(\chi_{D_1(0)}(z)+\chi_{D_2(0)}(z)),
\end{equation}
where the $D_m(t)$ are disjoint simply connected regions. Specifically, if the solution to \eqref{euler} takes the form 
\begin{equation}\label{ansatz}
\omega(t,z)=\omega_0(e^{-it\Omega}z),
\end{equation}
we get a rotating vortex pair $(D_1,D_2)$ about the origin $(0,0)$ with angular velocity $\Omega$. We choose the center of $D_1$ to be $(l,0)$ for some $l\in\mathbb{R}$ and we set
\begin{equation*}\label{symmetry}
D_2:=-D_1.
\end{equation*} 
We remark that $(l,0)$ is not a ``center of mass", in the sense that as we will see in Section~\ref{sect:formulation}, it is the center of the patch $D_1$ relative to the renormalization of the conformal maps.

Finally, we plug the ansatz \eqref{ansatz} in \eqref{euler} to get
\begin{equation}\label{euler_2}
\big(u_0(z)-\Omega z^\perp\big)\cdot n_{\partial D_m}=0,\qquad\text{for all }z\in\partial D_m,
\end{equation}  
for $m=1,2$ and where here, $n_{\partial D_m}$ denotes the unit normal vector to $\partial D_m$.

By moving to a frame of reference rotating at this same speed $\Omega$, the regions appear to be stationary. By expressing \eqref{euler} in terms of the relative stream function $\Psi=\psi_0-\tfrac{1}{2}\Omega|z|^2$ we then get
\begin{subequations}\label{elliptic formulation}
	\begin{alignat}{2}
	\Delta\Psi&=\frac{1}{\ep^2 \pi}\chi_{D_1}+\frac{1}{\ep^2 \pi}\chi_{D_2}-2\Omega,\\
	\nabla(\Psi+\tfrac{1}{2}\Omega|z|^2)&\rightarrow0,\quad\text{as }|z|\mapsto\infty\\
	\Psi&\in C^1(\mathbb{C}),\\\label{slov4}
	\Psi&=c_m,\qquad\text{on }\partial D_m,
	\end{alignat}
\end{subequations}
for some constants $c_m$, $m=1,2$. Since both the $D_m$ and the function $\Psi$ are unknowns, this is a free boundary problem.

\subsection{Main results}
The main result in this paper, is the following informally stated theorem. For a more rigorous statement, see Theorem \ref{thm:main_abstract}.
\begin{theorem}\label{thm:main}
	There exists a continuous curve $\mathscr C$ of corotating vortex patch solutions to \eqref{elliptic formulation}, parameterized by $s\in(0,\infty)$. Moreover, the following properties hold along $\mathscr C$:
	\begin{enumerate}[label=\rm(\roman*)]
		\item \textup{(Bifurcation from point vortex)} The solution at $s=0$ is a pair of points $z_1,z_2$ lying on the horizontal axis at a distance $l$ from each other, rotating with angular velocity $\Omega_0=1/(4\pi l^2)$.
		\item \textup{(Limiting configurations)}\label{thm:alternatives} As $s\to\infty$
		\begin{equation}\label{thm:physcial min}
			\min\bigg\{\min_{z\in\partial D_1}\ep \nabla\Psi(z)\cdot\bigg(\frac{z-l}{|z-l|}\bigg), \min_{z_m\in\partial D_m}|z_1-z_2| \bigg\}\to0
		\end{equation}
		\item \textup{($\ep$ bounded away from $0$)} The value of the parameter $\ep(s)$ is bounded away from $0$ for all $s$ away from the local curve.
		\item \textup{(Analyticity)} For each $s>0,$ the boundary $\partial D_m$ is analytic.
		\item \textup{(Graphical boundary)} For each $s>0$, the boundary of the patch can be expressed as a polar graph.
	\end{enumerate}
\end{theorem}
We briefly comment on the different alternatives in \ref{thm:alternatives}. The first term in \eqref{thm:physcial min} indicates that there are points $z(s)\in\partial D_1(s)$, (and by symmetry of the domains, hence also on $D_2$) for which the angular fluid velocity becomes arbitrarily small. We remark that the factor of $\ep$ is necessary in order to catch the domains in \eqref{elliptic formulation}, where the vorticity (for the purpose of the desingularization of point vortices) has been normalized to $1/(\pi\ep^2)$. Moreover, the slightly complicated formulation of angular fluid velocity comes from the fact that the patches are not centered at the origin. The formation of a corner or of a cusp would require that $\ep\nabla \Psi=0$ at a given point on the boundary of the patches and numerical evidence indicates that this does in fact happen. 

The second term in \eqref{thm:physcial min} vanishes if and only if the boundaries $\partial D_m$ of the two patches intersect at some point $z$. Clearly, this alternative can only occur in the situation of multiple patches. Numerical work \cite{overman:limiting} suggests that the limiting scenario consists of the two patches intersecting at a corner with a $90^\circ$ angle. This conjecture, also known as the \emph{Overman conjecture}, would imply that the two terms in \eqref{thm:physcial min} would occur simultaneously.

\subsection{Historical considerations}
The simplest explicit rotating vortex patch solution to \eqref{euler} is the disk, also referred to as the Rankine vortex. This solution was studied by Lord Kelvin \cite{thomson:columnar} at the linear level, who found that patches with boundary $r=1+\ep\cos(m\theta)$ rotate at constant angular velocity $\Omega_m=(m-1)/2m$, for $m=2,3,4\ldots$ Provided they have radial symmetry, more complicated explicit solutions can be found to the full nonlinear problem, such as for example a doubly-connected annulus. Note that the Rankine vortex can rotate at arbitrary angular velocity. This is in contrast to the ellipse-shaped solutions, found by Kirchhoff \cite{kirchhoff:book}, whose rotation depends on the parameters of the solution and specifically, on the eccentricity.

The initial results are numerical, and indicate that one can start from those trivial solutions and construct more complex ones from them. Notably, Deem and Zabusky \cite{dz:vstates} numerically showed the existence of polygonally shaped rotating patches with $m$-fold symmetry. Corotating and counter-rotating vortex pairs were investigated by Saffman and Szeto \cite{ss:pairs} and Pierrehumbert \cite{pierrehumbert}, who found that the patches were almost circular when far away from each other, but increasingly deform as they move closer to each other, until they eventually touch. Dritschel \cite{dritschel} numerically constructed pairs with different shapes and studied their linear stability. 

The first analytic existence results are perturbative, in the sense that they stem from an expansion around the explicit solutions, depending on some very small parameter $\ep$. The first rigorous existence proof for V-state solutions bifurcating from the Rankine vortex is due to Burbea \cite{burbea:motions}. He reformulated the elliptic problem in terms of conformal mappings and then obtained the local curve of solutions by means of the theorem of Crandall and Rabinowitz \cite{rabinowitz:simple}. This result was then improved first by Hmidi, Mateu and Verdera \cite{hmv:reg} who showed that the boundary of the solutions along the curve is smooth and convex, and then by Castro, Córdoba and Gómez-Serrano \cite{CCGS-2016-2} who showed that it is actually analytic. 

For the case when the explicit solutions are point vortices, the situation is more complicated. Marchioro and Pulvirenti \cite{mp:vortex} proved the desingularization of $N$-point vortices in the sense that they found smooth solutions $\omega_{\ep}$ to Euler's equations. Lamb \cite{lamb} discovered a nontrivial explicit example of touching counter-rotating pairs, where the vortex is not uniformly distributed but has a smooth compactly supported profile related to Bessel functions. Turkington \cite{turkington:nfold} analytically constructed pairs of corotating patches by restricting the attention to a fixed region around each patch and then solving a modified variational problem there. Although his approach is quite global, it does not yield sufficient structure of each vortex patch. Also of noteworthy mention are the results of Crowdy and Marshall \cite{cm:growing} and \cite{cm:analytical} in which the authors, by means of finding exact solutions, constructed a class of analytical solutions starting from the corotating point vortex pair which then "grew" into two symmetrical vortex patches at the two stagnation points in the corotating frame. Interestingly, as the area of these patches grew, they acquired "arms" of vorticity that not only eventually touch each other but form an exact circle of irrotational fluid enclosing the two point vortices.  

More recently, using Burbea's formulation and the implicit function theorem, Hmidi and Mateu \cite{hm:pairs} investigated the desingularization of point vortices by studying small vortex patches around each point. This was then extended by one of the authors to $N$-point vortices \cite{garcia:choreography}, or the K\'arm\'an Vortex Street \cite{garcia:karman}. Recently, Hassainia and Hmidi \cite{HH:asymmetric-pairs} studied the case of asymmetric pairs, and Hassainia and Wheeler \cite{Hassainia-Wheeler-points} proved the existence of more general multipole patch equilibria.

Despite the vast amount of work that has been done on perturbative results, very little has been done in the direction of existence proof for rotating patches that are outside a small neighborhood of the trivial solution. There is however strong numerical evidence for interesting solutions forming along the global curve. For instance, Wu, Overman and Zabusky \cite{woz:numerical} showed that the only plausible scenario for limiting V-states bifurcating from the Rankine vortex is the formation of corners with right angles. Likewise, based on numerical evidence, Overman \cite{overman:limiting} conjectured that the global curve for corotating and counter-rotating vortices must limit to the vortex patches intersecting at a right angle corner. 

The first and to the best of our knowledge, only, analytical global existence result for rotating patches is due to Hassainia, Masmoudi and Wheeler \cite{hmw:global} who analytically continued the local curve of solutions bifurcating from the Rankine vortex. By exploiting the analticity of the Cauchy integral operator and by reformulating the problem as a Riemann--Hilbert type problem, they used the global analytic bifurcation theorem initially due to Dancer \cite{dancer:global} and later improved by Buffoni and Toland \cite{bt:analytic}. This approach is inspired by the construction of large amplitude solutions to the steady two-dimensional water wave problem. Along this global curve, the solutions have an analytic boundary which can be expressed as a polar graph for an even, $2\pi/m$-periodic function and the limiting scenario consists of the vanishing of the angular fluid velocity.

Following the ideas in \cite{hmw:global}, we construct a global curve of solutions for corotating and counter-rotating vortex patches. However, our setting is more difficult than the one of the Rankine vortex. Indeed, our equations are overall more involved since we are dealing with more than just a single isolated patch. Moreover our formulation contains a singularity. Although this can be removed for the local theory, we must keep it in our formulation to be able to use ideas from Riemann--Hilbert theory. As a result, the traditional analytic bifurcation theorem is not applicable here as we cannot bifurcate directly from the trivial solution. Inspired by the theory of steady solitary water waves, we slightly modify the theorem to make it suitable to our needs. 

Finally we conclude this short presentation by pointing out that many of the above mentioned perturbation results have also been extended to other steady solutions for the Euler equations and also to other active scalar equations such as the surface quasi-geostrophic equations, see \cite{GHS:2020, GHM:2022, CCGS-2020, CCGS-2019, CCGS-2016, CCGS-2016-2, Dristchel-Hmidi-Renault, Hoz-Hassainia-Hmidi-Mateu:disc, Hmidi-Mateu:degenerate, Hmidi-Mateu:kirchhoff, Hoz-Hmidi-Mateu-Verdera:doubly-euler, Hoz-Hassainia-Hmidi:doubly-gSQG, Hassainia-Hmidi:vstates-gSQG, Hmidi-Mateu-Verdera:doubly, GSPSY-sheets, Roulley:2022, Hmidi-Roulley:2022, Hassainia-Roulley:2022, Hassainia-Hmidi-Masmoudi:2021, Berti-Hassainia-Masmoudi:2022}, and references therein.

\subsection{Outline of the paper}
We begin by carrying out the entire analysis for corotating vortex pairs. In Section~\ref{sect:formulation}, in the spirit of Burbea, we begin by reformulating our elliptic problem \eqref{elliptic formulation} in terms of conformal mappings. Using the Biot-Savart law and the Cauchy--Pompeiu formula, we then derive the contour integral equation for a pair of corotating vortex patches. Motivated by the approach in \cite{hmw:global}, we then express this equation as a scalar Riemann--Hilbert problem for $\phi'$, the derivative of the trace of the conformal map to the boundary of the disk. 

The main bulk of the paper lies in Section~\ref{sect:bounds} in which we provide quantitative bounds on the solution, depending only on lower bounds for the angular velocity field and on the fact that the two patches cannot self-intersect. These bounds are then of fundamental importance for the remainder of the paper. To begin with, they enable us to winnow out alternatives on the global curve in Section~\ref{sect:global}. Moreover, we also them to obtain results for the rigidity for corotating pairs. In the first place, we adapt in a straightforward way the rigidity theorem proved in \cite{GSPSY-rigidity} for a single patch, to our two-patch setting. That is, any corotating pair solution must have angular velocity $\Omega\in(0,\gamma/2)$, where $\gamma$ denotes the vortex strength. Otherwise the patches are radial and thus fail to be solutions to the problem. Since in our case the vorticity has been renormalized to include a factor of $1/\ep^2$, this bound is not sharp for small $\ep$. Thus, we complement this result with a new rigidity theorem, adapted to the specific structure of our solutions, which holds for all $\ep$ sufficiently small. Specifically, we show that under suitable circumstances, for sufficiently small values of $\ep$, the solutions must have a shape that is close to that of the disk, and that the perturbation $f$ as well as the angular velocity $\Omega$ must be bounded. The particularity of this result is that it holds along the entire global curve of solutions, and will then in turn be key in ruling out that $\ep$ can be small away from the local curve.

In Section~\ref{sect:global}, we construct our global curve of solutions for corotating pairs. After verifying that certain compactness criteria are satisfied, we extend the local curve (whose construction is recalled in Section~\ref{sect:local}) to a global one. Using a slightly modified version of the theorem of Buffoni and Toland, adapted to the fact that we have a singularity at the trivial solution, we obtain a global curve along which a blow-up type scenario occurs as the parameter $s\to\infty$. We then winnow down the various alternatives and prove that every solution along the global curve has an analytic boundary, thus obtaining a proof for Theorem~\ref{thm:main}. 

In Section~\ref{sect:translating}, we study counter-rotating pairs. We first briefly outline the formulation of the problem and then provide the main results. These are analogous to the ones for corotating patches and we briefly outline where the main differences occur. Finally, in an effort to keep the presentation as self-contained as possible, we provide an appendix which gathers the theorems needed for the construction of both the local and global curve, as well as some useful facts about Riemann--Hilbert problems and other cited results.

\subsection{Notation}
We recall the definitions of the Banach spaces which will be used throughout the paper. For any integer $k\geq0$ and for any $\alpha\in(0,1)$, we denote by $C^{k+\alpha}$ the space of functions whose partial derivatives up to order $k$ are Hölder continuous with exponent $\alpha$ over their domain of definition.\\

\begin{ackname}
The authors would like to thank M.H. Wheeler for proposing this interesting problem and for several discussions around it, and J. Park for conversations about the first rigidity theorem. The authors are also extremely grateful to the referees for their helpful comments and suggestions which have greatly contributed to improving the paper.
\end{ackname}

\section{Formulation}\label{sect:formulation}

\subsection{Point vortex model}
The $N$ point vortex system models the evolution of different point vortices located in the plane $z_{m,0}\in\R^2$ where $z_{m,0}\neq z_{i,0}$ and $m,i=1,...,N$ with $m\neq i$. That agrees with:
\begin{align*}
z_m'(t)=&\frac{1}{2\pi}\sum_{m\neq k=0}^{N}\Gamma_k\frac{(z_m(t)-z_k(t))^\perp}{|z_m(t)-z_k(t)|^2},\\
z_m(0)=&z_{m,0},
\end{align*}
here $\Gamma_m$ refers to the vortex strength.

We briefly provide the dynamics of two point vortices. Assuming that one of the points, $z_1$, initially lies on the horizontal axis, we have
\begin{equation}\label{points}
z_1(0)=l\quad\text{and}\quad z_2(0)=-z_1(0), 
\end{equation}  
for some $l\in\mathbb{R}$. Depending on the vortex strength $\Gamma_1$ and $\Gamma_2$ of each point vortex respectively, by \cite{newton:book}, the evolution of the two points consist of either a rotation of constant angular velocity or a translation at constant speed. That is, we have the following. 

\begin{proposition}
Consider two initial point vortices $z_1(0)$ and $z_2(0)$, with $z_1(0)\neq z_2(0)$, located in the real axis, with vortex strengths $\Gamma_1$ and $\Gamma_2$ respectively. Then:
\begin{itemize}
\item If $\Gamma_1+\Gamma_2\neq 0$ and $\Gamma_1  z_1(0)+\Gamma_2 z_2(0)=0$, then  $z_m(t)=e^{i\Omega_0 t}z_m(0)$, for $m=1,2$, with $\Omega_0=\frac{(\Gamma_1+\Gamma_2)}{2\pi |z_1(0)-z_2(0)|^2}$.
\item If $\Gamma_1+\Gamma_2=0$, then $z_m(t)=z_m(0)+V_0 t$, for $m=1,2$, with $V_0=\frac{i\Gamma_2}{2\pi} \textnormal{sign}(z_1(0)-z_2(0))/(|z_1(0)-z_2(0)).$
\end{itemize}
\end{proposition}
The proof of the above proposition is standard and can be found, for instance, in \cite{newton:book}.
We now take $z_1(0)$ and $z_2(0)$ as in \eqref{points}, normalize the vortex strength $\Gamma_1=1$ and let $\Gamma_2=a$ vary. If $a=1$ we have corotating point vortices with constant angular velocity $\Omega_0$:
\begin{equation}\label{Omega0}
z_m(t)=e^{i\Omega_0 t}z_m(0), \qquad\text{where }\Omega_0=\frac{1}{4\pi l^2},
\end{equation}
for $m\in\{1,2\}$. On the other hand, if $a=-1$, we obtain counter-rotating vortices:
\begin{equation*}\label{V0}
z_m(t)=z_m(0)+V_0 t, \qquad\text{where }V_0=-\frac{i}{4\pi l}.
\end{equation*}

\subsection{Conformal mapping}
In order to fix the boundary of the domain, we use a conformal map $\Phi$ which maps the exterior of the unit disk $\mathbb{D}$ to the exterior of the centered patch $D:=(D_1-l)/\ep$. We normalize the map by requiring that $\Phi(\infty)=\infty$ and in order to avoid invariance under translations, set $\Phi'(\infty)=1$. The map is constructed such that the boundary $\mathbb{T}:=\partial\mathbb{D}$ is mapped to $\partial D$. We consider solutions to the problem \eqref{elliptic formulation} for which $D$ is of class $C^{k+\alpha}$ for some $k\geq0$. By the Kellogg--Warschawski theorem \cite[Theorem~3.6]{pommerenke:book}, we conclude that $\Phi$ is of class $C^{k+\alpha}(\mathbb{C}\setminus\mathbb{D})$ and its trace $\phi:=\Phi\big|_\mathbb{T}$ is a $C^{k+\alpha}$ parametrization of $\partial D$. 

We choose the conformal map such that 
\begin{equation}\label{conformal}
\partial D_1=\varepsilon\phi(\mathbb{T})+l,\qquad\text{and}\qquad \partial D_2=-\varepsilon\phi(\mathbb{T})-l,
\end{equation}
where $\phi$ is such that
\begin{equation}\label{slov5}
	\phi(w)=w+\ep f(w),\qquad f(w)=\sum_{n\geq 1}a_n w^{-n},\qquad\text{with }f\in C^{k+\alpha}(\mathbb{T}),
\end{equation}
for $a_n\in\mathbb{R}$ and $w\in\T$. By considering $a_n\in\R$ we have assumed that the domain is symmetric with respect to the $x$-axis. 


Since for most of the paper we will work with the centered, renormalized patch $D$, it will be useful to have a reformulation of the elliptic problem in terms of the associated relative stream function $\tilde{\Psi}$. We set $\tilde{\Psi}(z)=\Psi(l+\ep z)$ and get
\begin{subequations}\label{elliptic formulation-centered}
	\begin{alignat}{2}
	\Delta\tilde{\Psi}&=\frac{1}{\pi}\chi_{D(z)}+\frac{1}{\pi}\chi_{D(-z-2l/\ep)}-2\ep^2\Omega,\\\label{grad_Psi}
	\nabla\big(\tilde{\Psi}&+\tfrac{1}{2}\ep^2\Omega|z|^2+\Omega\ep l\cdot z\big)\rightarrow0,\quad\text{as }|z|\mapsto\infty\\
	\tilde{\Psi}&\in C^1(\mathbb{C}),\\\label{slov4-centered}
	\tilde{\Psi}&=c,\qquad\text{on }\partial D,
	\end{alignat}
\end{subequations}
for some constant $c$.

\subsection{Velocity formulation}
The velocity can be recovered from the vorticity according to the Biot-Savart law. Specifically
\begin{equation*}
\psi_0(z)=\frac{1}{2\pi^2\ep^2}\int_{D_1}\log|z-\zeta|\,d\zeta+\frac{1}{2\pi^2\ep^2}\int_{D_2}\log|z-\zeta|\,d\zeta.
\end{equation*}
From the definition of $u$ in \eqref{euler} we get
\begin{equation*}
u_0(z)=\nabla^\perp\psi_0(z)=\frac{i}{2\pi^2\ep^2}\int_{D_1}\frac{d\zeta}{\overline{z-\zeta}}\,d\zeta+\frac{i}{2\pi^2\ep^2}\int_{D_2}\frac{d\zeta}{\overline{z-\zeta}}.
\end{equation*}
By using the Cauchy--Pompeiu formula, we obtain
\begin{equation}\label{c-p}
u_0(z)=\frac{1}{4\pi^2\ep^2}\overline{\int_{\partial D_1}\frac{\overline{z}-\overline{\zeta}}{z-\zeta}}\,d\zeta+\frac{1}{4\pi^2\ep^2}\overline{\int_{\partial D_2}\frac{\overline{z}-\overline{\zeta}}{z-\zeta}}\,d\zeta.
\end{equation}
Moreover, from \eqref{initial_vorticity} and \eqref{conformal} we get the conformal parametrization of the velocity
\begin{equation*}
	u_0(\ep\phi(w)+l)=\frac{1}{4\pi^2\ep}\overline{\int_{\T}\frac{\overline{\phi(w)}-\overline{\phi(\zeta)}}{\phi(w)-\phi(\zeta)}\phi'(\zeta)\, d\zeta}-\frac{1}{4\pi^2\ep}\overline{\int_{\T}\frac{\ep\overline{\phi(w)}+\ep\overline{\phi(\zeta)}+2l}{\ep\phi(w)+\ep\phi(\zeta)+2l}\phi'(\zeta)\, d\zeta},
\end{equation*}
and finally, after some calculations, we see that \eqref{euler_2} can be written as
\begin{equation}\label{euler_im}
	\Img\bigg\{\bigg(\frac{1}{i}\overline{u\big(\ep\phi(w)+l\big)}+\Omega\big(\ep\overline{\phi(w)}+l\big)\bigg)w\phi'(w) \bigg\}=0.
\end{equation}
\subsection{Formulation with singularity in $\varepsilon$: Cauchy integral}
Following the ideas in \cite{hmw:global}, we express the nonlinear terms in \eqref{euler_im} in a nonlocal formulation. This will be useful for the construction of the global continuation of the curve of solutions. To this end, we introduce the Cauchy integral operator $\mathcal{C}(\phi)$
\begin{equation}\label{Ccauchy}
	\mathcal{C}(\phi)\colon g\mapsto\frac{1}{2\pi i}\int_{\T}\frac{g(\tau)-g(w)}{\phi(\tau)-\phi(w)}\phi'(\tau)\,d\tau,
\end{equation}
associated to the curve $\partial D=\phi(\T)$. In addition, we set
\begin{equation*}\label{Ctilde}
	\tilde{\mathcal{C}}_{\ep,l}(\phi)g=\frac{1}{2\pi i}\int_{\T}\frac{\ep g(\tau)+\ep g(w)+2l}{\ep\phi(\tau)+\ep\phi(w)+2l}\phi'(\tau)\,d\tau.	
\end{equation*}
We now rewrite \eqref{euler_im} as
\begin{equation}\label{cauchy_formulation}
	\Img\bigg\{\bigg(\frac{1}{2\pi\ep}\mathcal{C}(\phi)\overline{\phi}-\frac{1}{2\pi \ep}\tilde{\mathcal{C}}_{\ep,l}(\phi)\overline{\phi}+\Omega\big(\ep\overline{\phi}+l\big)\bigg)w\phi' \bigg\}=0.	
\end{equation}
Defining $A$ as
\begin{equation}\label{A}
A:=\bigg(\frac{1}{2\pi\ep}\mathcal{C}(\phi)\overline{\phi}-\frac{1}{2\pi\ep}\tilde{\mathcal{C}}_{\ep,l}(\phi)\overline{\phi}+\Omega\big(\ep\overline{\phi}+l\big)\bigg)w.
\end{equation}
we now get the following reformulation of \eqref{elliptic formulation-centered} in terms of the conformal parametrization of the problem:
\begin{subequations}\label{A_formulation}
	\begin{equation}
	\Img\big\{A\phi' \big\}=0,
	\end{equation}
	with the additional assumption that
	\begin{equation}
		|\ep A|>0\qquad\text{for }\ep\in\mathbb{R}^+.
	\end{equation}
	We require $\phi$ to have the regularity
	\begin{equation}
		\phi\in C^{k+\alpha}(\mathbb{T}),
	\end{equation}
	to satisfy the property
	\begin{equation}\label{non_intersect}
		\inf_{\tau, w}\big|\ep\phi(\tau)+\ep\phi(w)+2l\big|>0
	\end{equation}
	and to be of the form
	\begin{equation}\label{polar graph}
		\phi(e^{it})=\rho(t)e^{i\vartheta(t)},\quad\text{with }\vartheta'>0\quad\text{for }t\in\mathbb{R},
	\end{equation}	
\end{subequations}
where $\rho>0$ and $\vartheta$ are periodic real-valued functions in $C^{k+\alpha}$. We remark that \eqref{polar graph} indicates that $\phi(\mathbb{T})$ is a polar graph $r=R(\theta)$ for some $C^{k+\alpha}$ function $R$, (see \cite[Lemma 2.4]{hmw:global}).

\begin{remark}[Analyticity of the Cauchy integral]\label{rem:analytic}
	Of particular mention is the fact that any $\phi\in C^{k+\alpha}(\mathbb{T})$ satisfying \eqref{polar graph} also satisfies the condition
	\begin{equation}\label{injectivity}
	\inf_{\tau\neq w}\bigg|\frac{\phi(\tau)-\phi(w)}{\tau-w} \bigg|>0.
	\end{equation}
	This was proven in \cite[Lemma 2.6]{hmw:global}. The condition \eqref{injectivity} ensures that $\phi$ is injective and that $\phi'\neq0$. This not only guarantees that we have an equivalence between \eqref{elliptic formulation-centered} and \eqref{cauchy_formulation} and is a necessary requirement for $\phi$ to be able to be extended to a conformal map $\Phi$ on $\mathbb{C}\setminus\mathbb{D}$, but also ensures that the Cauchy integral operator $\mathcal{C}(\phi)$ \eqref{Ccauchy} is real-analytic in $\phi$. This result has been proven in \cite{ldcl:cauchy} and is also recalled in \cite{hmw:global}. Moreover, the property \eqref{non_intersect}, necessary assumption to ensure that the two patches $D_1$ and $D_2$ cannot intersect, yields by similar arguments as in \cite{ldcl:cauchy} that the operator $\tilde{\mathcal{C}}_{\ep,l}(\phi)$ must also be real-analytic in $\phi$. 
\end{remark}

\subsection{Formulation without singularity in $\varepsilon$}
The formulation \eqref{cauchy_formulation} contains a singularity at $\ep=0$. Although this will prove to not be a major issue for the global bifurcation, when constructing the local curve, we need to consider the trivial solutions, that is, the case when $\ep=0$. Thankfully, the singularity can be removed. This process was carried out in \cite{hm:pairs} and \cite{garcia:choreography}. However, for the sake of completeness we provide the key steps of the calculation. 

To begin with, notice that
\begin{equation*}
	\int_{\partial D_2}\frac{1}{z-\zeta}\,d\zeta=0,\qquad\text{if }z\in\partial D_1.
\end{equation*}
We can therefore rewrite \eqref{c-p} as 
\begin{equation*}
u_0(z)=\frac{1}{4\pi^2\ep^2}\overline{\int_{\partial D_1}\frac{\overline{z}-\overline{\zeta}}{z-\zeta}}\,d\zeta-\frac{1}{4\pi^2\ep^2}\overline{\int_{\partial D_2}\frac{\overline{\zeta}}{z-\zeta}}\,d\zeta.
\end{equation*}
As a result, we easily get
\begin{equation*}\label{u_0}
	u_0(\ep\phi(w)+l)=\frac{1}{4\pi^2\ep}\overline{\int_{\T}\frac{\overline{\phi(w)}-\overline{\phi(\zeta)}}{\phi(w)-\phi(\zeta)}\phi'(\zeta)\,d\zeta}-\frac{1}{4\pi^2}\overline{\int_{\T}\frac{\overline{\phi(\zeta)}\phi'(\zeta)}{\ep\phi(w)+\ep\phi(\zeta)+2l}\,d\zeta}.
\end{equation*}
Using the fact that we chose $\phi(w)=w+\ep f(w)$, we can now express \eqref{euler_im} as
\begin{equation}\label{formulation-without-epsilon}
	\frac{1}{2\pi}\textnormal{Im}\left[f'(w)\right]+\textnormal{Im}\left[\left\{\mathscr J(f,\ep)(w)-\Omega(\ep \overline{\phi(w)}+l)\right\}w\phi'(w)\right]=0,
\end{equation}
where
\begin{align}\label{J}
\mathscr J(f,\ep)(w)=&\frac{i}{4\pi^2}\int_{\T} \frac{\overline{\phi(w)-\phi(\xi)}}{\phi(w)-\phi(\xi)}f'(\xi)d\xi-\frac{1}{2\pi^2}\int_{\T} \frac{\textnormal{Im}\left[(w-\xi)(f(\overline{w})-f(\overline{\xi}))\right]}{(w-\xi)(\phi(w)-\phi(\xi))}d\xi\nonumber\\
&-\frac{i}{4\pi^2}\int_{\T} \frac{\overline{\phi(\xi)}}{\varepsilon \phi(\xi)+\varepsilon\phi(w)+2l}\phi'(\xi)d\xi.
\end{align}

Let us remark that the last integral in $\mathscr{J}(f,\ep)$ refers to the interaction between the left hand side domain and the right hand side domain. Indeed, the denominator in this term is simply the distance between a point on $\partial D_1$ and a point on $\partial D_2$. If the two domains are well-separated, this denominator never vanishes thus precluding the formation of any undesired singularity.

\section{Quantitative bounds}\label{sect:bounds}
This section is devoted to proving certain quantitative bounds on the solutions to the problem \eqref{elliptic formulation-centered}. Indeed we can show that many quantities can be bounded by $\delta>0$ in the inequalities
\begin{equation}\label{delta}
	\big|\partial_r\tilde{\Psi}\big|\geq\delta\quad\text{on }\partial D\qquad\text{and}\qquad \big|\ep\phi(\tau)+\ep\phi(w)+2l\big|\geq\delta.
\end{equation}  
The first term in \eqref{delta} assumes that the angular velocity of the patch is uniformly bounded away from $0$  and the second term is a necessary condition such that the two patches in the physical domain do not self-intersect with each other. We will use the notation $\lesssim_\delta$ whenever a constant depends on $\delta$. 

The first step will be collecting some rigidity theorems. While also very interesting on their own, they will prove to be very useful tools throughout the paper.

\subsection{Rigidity theorem}
There are a number of results concerning the rigidity for rotating single case. First, Fraenkel \cite{Fraenkel} showed that the stationary ($\Omega=0$) single patches must to be radial. The same result was obtained by Hmidi \cite{hmidi:trivial} for $\Omega\leq 0$ and $\Omega=\frac12$. Finally, G\'omez-Serrano, Park, Shi and Yao \cite{GSPSY-rigidity} proved that a rotating single patch (with vorticity normalized to $1$ inside the patch) with $\Omega\in(-\infty,0]\cup [\frac12,+\infty)$ must be radial. They use a clever idea involving the first variation of some appropriate energy functional. In Remark 2.3 of \cite{GSPSY-rigidity} they state that a similar proof also works in the case of disconnected patches as long as each connected component is simply--connected, which is the case of the corotating patches. Hence, their result adapted to our framework says that any corotating patches whose vorticity is normalized to $1$ should have $\Omega\in(0,\frac12)$. For the sake of completeness, we give here the proof for non-normalized corotating patches since the vorticity inside the desingularization patches is $\frac{1}{\pi \ep^2}$. The proof follows along the same lines as in \cite{GSPSY-rigidity}.

In the second part of the theorem, we prove a new rigidity theorem which holds for all $\ep$ sufficiently small. Although the first bound in Theorem \ref{thm:rigidity1} is useful for proving the analyticity of the boundary of the patches, the second part of the theorem will be very important in showing that $\ep$ can only go to $0$ along the local bifurcation curve. 

\begin{theorem}\label{thm:rigidity1}
	(i) If $\omega_0=(\pi\ep^2)^{-1}(\chi_{D_1}+\chi_{D_2})$ is a solution of the system \eqref{elliptic formulation} then
	\begin{equation}\label{slov10}
	\Omega\in\Big(0,\frac{1}{2\pi\ep^2}\Big].
	\end{equation}
	
	(ii) Moreover, if $\ep\leq l/10$ then
	\begin{equation}\label{end1}
	|\Omega|\lesssim |l|^{-2},
	\end{equation}
	where here $l$ denotes the distance of the center of the patch $D_1$ to the $y$-axis.
\end{theorem}

\begin{proof}
	(i) The proof is inspired by the recent paper \cite{GSPSY-rigidity}. We define
	$$
	I:=-\int_{D_1}v_1\cdot \nabla f_\Omega dx,
	$$
	where
	\begin{equation}\label{slov11}
	f_\Omega(x):=\frac{1}{\pi\ep^2} \chi_{D_1}\star \mathcal{N}+\frac{1}{\pi\ep^2} \chi_{D_2}\star \mathcal{N}-\frac{\Omega}{2}|x|^2,
	\end{equation}
	and $\mathcal{N}(z)=(2\pi)^{-1}\log |z|$ is the Newtonian potential. By \eqref{slov4} we have
	\begin{equation}
	\label{end3}
	\begin{split}
	&f_\Omega(x)=c_1 \quad \textnormal{on }\partial D_1,\qquad \Delta f_\Omega=\frac{1}{\pi\ep^2}-2\Omega\quad \textnormal{in } D_1.
	\end{split}
	\end{equation}
	We remark that the function $f_\Omega$ is actually the relative stream function  $\Psi_1$ of associated to the patch $D_1$.
	
	We consider vector fields $v_1=-\nabla\varphi_1$, with 
	\begin{equation}\label{slov6}
	\begin{split}
	&\varphi_1(x)=\frac{|x|^2}{2}+p_1(x) \quad \text{in } D_1,\\
	&p_1(x):=\ep^2 p\Big(\frac{x-l}{\ep}\Big),
	\end{split}
	\end{equation}
	where the function $p$ is defined  by the conditions
	\begin{equation}\label{p0}
	\Delta p(x)=-2 \quad \textnormal{ in }  {D},\qquad\qquad p(x)=0\quad \textnormal{ on } \partial {D},
	\end{equation}
	In particular,
	\begin{align*}\label{pm}
	\Delta p_1(x)=-2 \quad \textnormal{in }  {D}_1,\qquad p_1(x)=0 \quad \textnormal{on } \partial {D}_1.
	\end{align*}
	The key point is that the function $p:D\to\R$ is positive and satisfies the inequality (see \cite{GSPSY-rigidity})
	\begin{equation}\label{JaviPa}
	\frac{1}{4\pi}|D|^2-\int_{D}p(y) dy\geq 0,
	\end{equation}
	with equality only if $D$ is a disk. 
	
	Since $v_1$ is divergence free and $f_\Omega=c_1$ on $\partial D_1$ (see \eqref{end3}) we have 
	\begin{equation}\label{I=0}
	I=-\int_{D_1}v_1\cdot \nabla (f_\Omega-c_1) dx=0.
	\end{equation}
	On the other hand, using the expression of $v_1$ we can write $I$ as follows
	\begin{align*}
	I=&\int_{D_1}x\cdot \nabla f_\Omega dx+\int_{D_1}\nabla p_1(x)\cdot \nabla f_\Omega dx= I_1+I_2+I_3,
	\end{align*}
	where
	\begin{align*}
	I_1:=\frac{1}{\pi \ep^2}\int_{D_1}x\cdot \nabla (\mathcal{N}\star \chi_{D_1}) dx+\frac{1}{\pi \ep^2}\int_{D_1}x\cdot \nabla (\mathcal{N}\star \chi_{D_2}) dx,
	\end{align*}
	\begin{align*}
	I_2:=-\Omega \int_{D_1}|x|^2dx,
	\end{align*}
	\begin{align*}
	I_3:=\int_{D_1}\nabla p_1(x)\cdot \nabla f_\Omega dx.
	\end{align*}
	We calculate now two of these three integrals. We have
	\begin{equation*}
	\begin{split}
	I_1&=\frac{1}{2\pi^2 \ep^2}\int_{D_1}\int_{D_1}x\cdot \frac{x-y}{|x-y|^2} dydx+\frac{1}{2\pi^2 \ep^2}\int_{D_1}\int_{D_2}x\cdot \frac{x-y}{|x-y|^2} dydx\\
	&=\frac{1}{4\pi^2 \ep^2}\int_{D_1}\int_{D_1}(x-y)\cdot \frac{x-y}{|x-y|^2} dydx+\frac{1}{4\pi^2 \ep^2}\int_{D_1}\int_{D_1}(x+z)\cdot \frac{x+z}{|x+z|^2} dzdx\\
	&=\frac{\ep^2}{2\pi^2}|D|^2,
	\end{split}
	\end{equation*}
	\begin{equation*}
	\begin{split}
	I_3=-\int_{D_1}p_1(x)\cdot \Delta f_\Omega dx=-\ep^2\int_{D}\ep^2p(y)\cdot \Big(\frac{1}{\pi\ep^2}-2\Omega\Big)dy=-\frac{\ep^2}{\pi}\int_{D}p(y) dy+2\Omega\ep^4\int_{D}p(y) dy.
	\end{split}
	\end{equation*}
	Moreover, by the standard rearrangement inequality,
	\begin{equation*}
	\int_{D_1}|x|^2dx=\frac{1}{2}\int_{D_1\cup D_2}|x|^2dx\geq \frac{1}{4\pi}|D_1\cup D_2|^2=\frac{1}{\pi}|D_1|^2=\frac{\ep^4}{\pi}|D|^2.
	\end{equation*}
	
	We are now ready to prove the desired conclusion \eqref{slov10}. If $\Omega\leq 0$ then
	\begin{equation*}
	\begin{split}
	I_1+I_2+I_3&\geq \frac{\ep^2}{2\pi^2}|D|^2+|\Omega|\frac{\ep^4}{\pi}|D|^2-\frac{\ep^2}{\pi}\int_{D}p(y) dy-2|\Omega|\ep^4\int_{D}p(y) dy\\
	&=\frac{\ep^2}{\pi}\Big[\frac{1}{2\pi}|D|^2-\int_{D}p(y) dy\Big]+2|\Omega|\ep^4\Big[\frac{1}{2\pi}|D|^2-\int_{D}p(y) dy\Big]\\
	&\geq \frac{\ep^2}{4\pi^2}|D|^2,
	\end{split}
	\end{equation*}
	using \eqref{JaviPa} in the last line. This contradicts the identity $I=0$ in \eqref{I=0}. On the other hand, if $\Omega> (2\pi\ep^2)^{-1}$ then
	\begin{equation*}
	\begin{split}
	I_1+I_2+I_3&\leq \frac{\ep^2}{2\pi^2}|D|^2-\Omega\frac{\ep^4}{\pi}|D|^2-\frac{\ep^2}{\pi}\int_{D}p(y) dy+2\Omega\ep^4\int_{D}p(y) dy\\
	&=\frac{\ep^2}{\pi}\Big[\frac{1}{2\pi}|D|^2-\int_{D}p(y) dy\Big]-2\Omega\ep^4\Big[\frac{1}{2\pi}|D|^2-\int_{D}p(y) dy\Big]\\
	&\leq-\frac{1}{4\pi}|D|^2\Big(2\Omega\ep^4-\frac{\ep^2}{\pi}\Big),
	\end{split}
	\end{equation*}
	using again \eqref{JaviPa} in the last line. This contradicts the identity $I=0$ in \eqref{I=0}, which completes the proof of \eqref{slov10}.
	
	(ii) As before we define 
	$$
	J:=-\int_{D_1}w_1\cdot \nabla f_\Omega dx,
	$$
	where $f_\Omega$ is defined as in \eqref{slov11}, $w_1:=-\nabla\varphi_1$, and $\varphi_1$ is defined by
	\begin{equation*}\label{slov12}
	\varphi_1(x)=\frac{|x-q_1|^2}{2}+p_1(x) \quad \text{in } D_1,\qquad q_1=l+\ep\widetilde{q}.
	\end{equation*}
	The functions $p$ and $p_1$ are defined as in \eqref{slov6}--\eqref{p0}. The choice of the point $q_1=l+\ep\widetilde{q}$ is important; we will make this choice later, satisfying $|\widetilde{q}|\lesssim 1$.
	
	As in \eqref{I=0} we see that $J=0$. On the other hand, using the expression of $w_1$ we can write $J$ as follows
	\begin{align*}
	J=&\int_{D_1}(x-q_1)\cdot \nabla f_\Omega dx+\int_{D_1}\nabla p_1(x)\cdot \nabla f_\Omega dx= J_1+J_2+J_3+J_4,
	\end{align*}
	where
	\begin{align*}
	J_1:=\frac{1}{\pi \ep^2}\int_{D_1}(x-q_1)\cdot \nabla (\mathcal{N}\star \chi_{D_1}) dx,
	\end{align*}
	\begin{align*}
	J_2:=\frac{1}{\pi \ep^2}\int_{D_1}(x-q_1)\cdot \nabla (\mathcal{N}\star \chi_{D_2}) dx,
	\end{align*}
	\begin{align*}
	J_3:=-\Omega \int_{D_1}(x-q_1)\cdot xdx,
	\end{align*}
	\begin{align*}
	J_4:=\int_{D_1}\nabla p_1(x)\cdot \nabla f_\Omega dx.
	\end{align*}
	We estimate now these four integrals. Assuming that $\ep\ll |l|$ and $|\widetilde{q}||\lesssim 1$ we have
	\begin{equation*}
	\begin{split}
	J_1&=\frac{1}{2\pi^2 \ep^2}\int_{D_1}\int_{D_1}(x-q_1)\cdot \frac{x-y}{|x-y|^2} dydx\\
	&=\frac{1}{4\pi^2 \ep^2}\int_{D_1}\int_{D_1}[(x-q_1)-(y-q_1)]\cdot \frac{(x-q_1)-(y-q_1)}{|(x-q_1)-(y-q_1)|^2} dydx\\
	&=\frac{\ep^2}{4\pi^2}|D|^2,
	\end{split}
	\end{equation*}
	\begin{equation*}
	\begin{split}
	|J_2|&=\Big|\frac{1}{2\pi^2 \ep^2}\int_{D_1}\int_{D_2}(x-q_1)\cdot \frac{x-y}{|x-y|^2} dydx\Big|\lesssim\frac{1}{2\pi^2 \ep^2}\int_{D_1} |x-q_1| \frac{1}{|l|}|D_2|dx\lesssim \frac{\ep^3}{|l|}|D|^2,
	\end{split}
	\end{equation*}
	\begin{equation*}
	\begin{split}
	J_3=-\Omega\ep^2\int_{D}(l+\ep y-q_1)\cdot (l+\ep y)dy=-\Omega\ep^3\int_{D}(y-\widetilde{q})\cdot (l+\ep y)dy,
	\end{split}
	\end{equation*}
	\begin{equation*}
	\begin{split}
	J_4=-\int_{D_1}p_1(x)\cdot \Delta f_\Omega dx=-\ep^2\int_{D}\ep^2p(y)\cdot \Big(\frac{1}{\pi\ep^2}-2\Omega\Big)dy=-\frac{\ep^2}{\pi}\int_{D}p(y) dy+2\Omega\ep^4\int_{D}p(y) dy.
	\end{split}
	\end{equation*}
	The main point is that we can now find $\widetilde{q}$ with $|\widetilde{q}|\lesssim 1$ such that $J_3\in[|\Omega|\ep^3|D||l|,2|\Omega|\ep^3|D||l|]$. This is important to obtain the second term in left hand side of \eqref{slov9} which will then ultimately yield the second inequality in \eqref{slov8}. 
	
	Recall the main inequality \eqref{JaviPa}. 
	Since $J=0$ we have $|J_1+J_3+J_4|\leq |J_2|$, therefore
	\begin{equation}\label{slov9}
	\Big[\frac{\ep^2}{4\pi^2}|D|^2-\frac{\ep^2}{\pi}\int_{D}p(y) dy\Big]+|\Omega|\ep^3|D||l|\lesssim \frac{\ep^3}{|l|}|D|^2+|\Omega|\ep^4\int_{D}p(y) dy\lesssim \frac{\ep^3}{|l|}|D|^2+|\Omega|\ep^4|D|^2.
	\end{equation}
	If $\ep/|l|\ll 1$ then the last term in the right-hand side can be absorbed into the left-hand side, thus
	\begin{equation*}
	\Big[\frac{\ep^2}{4\pi^2}|D|^2-\frac{\ep^2}{\pi}\int_{D}p(y) dy\Big]+|\Omega|\ep^3|D||l|/2\lesssim \frac{\ep^3}{|l|}|D|^2,
	\end{equation*}
	which gives our main conclusion
	\begin{equation}\label{slov8}
	\begin{split}
	&\Big[\frac{1}{4\pi}|D|^2-\int_{D}p(y) dy\Big]\lesssim \frac{\ep}{|l|}|D|^2,\\
	&|\Omega||l|^2\lesssim 1.
	\end{split}
	\end{equation}
	In particular, this completes the proof of \eqref{end1}. The inequality in the first line of \eqref{slov8} is also useful and will be exploited later in Lemma \ref{thm:rigidity2} to prove that the domain $D$ must be close to a disk as $\ep\to 0$. 
\end{proof}

A consequence of the above theorem is the following bound on the gradient of $\tilde{\Psi}$.

\begin{corollary}\label{lem:elliptic_estimates}
	Suppose that $(\tilde{\Psi},\Omega,\ep)$ solves \eqref{elliptic formulation-centered}. Then there exists a constant $C$ depending only on $\beta$ such that
	\begin{equation}\label{grad_bounds_tilde_Psi}
		\|\partial_z\tilde{\Psi}\|_{C^\beta(\overline{D})}\leq C.
	\end{equation}
\end{corollary}
\begin{proof}
	We begin by noticing that the function $\tilde{\Psi}_1:=\tilde{\Psi}+\tfrac12\ep^2\Omega|z|^2+\Omega\ep l\cdot z+\tfrac12\Omega l^2$ satisfies
	\begin{equation*}
		\Delta\tilde{\Psi}_1=\frac{1}{\pi}\chi_{D(z)}+\frac{1}{\pi}\chi_{D(-z-2l/\ep)}.
	\end{equation*} 
	Moreover, from \eqref{grad_Psi}, we see that $\nabla\tilde{\Psi}_1\to0$ as $|z|\to\infty$. By standard elliptic theory, see for instance \cite{GT}, we thus get
	\begin{equation*}
		\|\partial_z\tilde{\Psi}_1\|_{C^\beta(\overline{D})}\leq C_1,
	\end{equation*}
	where the constant $C_1$ only depends on $\beta$. Moreover, from the rigidity result Theorem~\eqref{thm:rigidity1}, we have bounds on $\ep^2\Omega$ for all $\ep$ and on $\Omega l$ for $\ep$ sufficiently small, which yields the desired result \eqref{grad_bounds_tilde_Psi}. 
\end{proof}

\subsection{Main quantitative bounds}
In this section, we will use \eqref{delta} and the results from the previous subsection to bound important quantities in our problem. We begin by showing that, provided \eqref{delta} holds, we get uniform lower bounds on the relative velocity field for the centered renormalized patch $D$, as well as can ensure that the boundary $\partial D$ is graphical. The proofs in this section closely follow the arguments in \cite{hmw:global}.

\begin{lemma}\label{lem:Psi_r}
	Suppose that $(\tilde{\Psi},\Omega,\ep)$ solves \eqref{elliptic formulation-centered} and that \eqref{delta} holds. Then
	\begin{equation}\label{bounds on A and arg}
		|\ep A|\gtrsim_\delta1\qquad\text{and}\qquad\bigg|\arg\bigg(\frac{w\phi'}{\phi}\bigg)-\frac{\pi}{2}\bigg|\gtrsim_\delta1,
	\end{equation}
	where here $A$ is defined as in \eqref{A}.
\end{lemma}
\begin{proof}
	If \eqref{delta} holds, then we directly obtain that
	\begin{equation*}
		\frac{|\nabla\tilde{\Psi}|^2}{4}\geq\frac{\delta^2}{4}.
	\end{equation*}
	Since we can see from Section~\ref{sect:formulation} that $\ep A$ is directly linked to the gradient of $\tilde{\Psi}$, we immediately get the first uniform lower bound in \eqref{bounds on A and arg}. 
	
	For the second one, we differentiate $\tilde{\Psi}\circ\phi\equiv0$. This yields
	\begin{equation}\label{differentiation}
	|\phi|\operatorname{tan}\bigg(\operatorname{arg}\frac{w\phi'}{\phi}\bigg)=\frac{|\phi|\Img\big(w\phi'/\phi\big)}{\operatorname{Re}\big(w\phi'/\phi\big)}.
	\end{equation}
	Since $\phi$ is a polar graph (see Section~\ref{sect:formulation}), we can write
	\begin{equation*}
	\phi(e^{it})=\rho(t)e^{i\varphi(t)},
	\end{equation*}
	for $t\in\R$. A direct calculation shows that
	\begin{equation*}
	\varphi'(t)=\operatorname{Re}\bigg(\frac{e^{it}\phi'(e^{it})}{\phi(e^{it})}\bigg)\qquad\text{and}\qquad \rho'(t)=-\rho(t)\Img\bigg(\frac{e^{it}\phi'(e^{it})}{\phi(e^{it})}\bigg).
	\end{equation*}
	Plugging this into the last term in \eqref{differentiation} gives us
	\begin{equation*}
	\frac{|\phi|\Img\big(w\phi'/\phi\big)}{\operatorname{Re}\big(w\phi'/\phi\big)}=-\frac{\rho'}{\varphi'}=\frac{\tilde{\Psi}_\theta}{\tilde{\Psi}_r}.
	\end{equation*}
	We thus deduce that
	\begin{equation}\label{tan_bounds}
	\tan\bigg(\bigg\|\arg\frac{w\phi'}{\phi}\bigg\|_{L^\infty(\T)}\bigg)=\bigg\|\frac{\tilde{\Psi}_\theta/\rho}{\tilde{\Psi}_r}\bigg\|_{L^\infty(\partial\Omega)}\leq\frac{\|\tilde{\Psi}_\theta/\rho\|_{L^\infty(\partial D)}}{\delta}\lesssim_\delta1,
	\end{equation}
	where the last step follows from the fact that Lemma~\ref{lem:elliptic_estimates} ensures that $\|\partial_z\tilde{\Psi}\|_{C^{1/2}(\overline{D})}\leq C$. From \eqref{tan_bounds}, we hence get
	\begin{equation*}
	\max_{\T}\bigg|\operatorname{arg}\frac{w\phi'}{\phi} \bigg|\gtrsim_\delta\frac{\pi}{2},
	\end{equation*} 
	thus concluding the proof. 
\end{proof}

From the previous lemma, we now see that \eqref{delta} can be rewritten as 
\begin{equation}\label{uniform_bounds_delta}
\bigg\|\operatorname{arg}\frac{w\phi'}{\phi} \bigg\|_{L^\infty}<\frac{\pi}{2}-\delta,\qquad\text{and}\qquad\inf_{\mathbb{T}}|\ep A|>\delta\qquad\text{and}\qquad \big|\ep\phi(\tau)+\ep\phi(w)+2l\big|\geq\delta.
\end{equation}

The next two statements provide us with lower bounds on $\phi$ and $\phi'$ as well as uniform bounds on $\phi$ in $C^{1+\alpha}$, all dependent on $\delta$ from \eqref{uniform_bounds_delta}. These results will have a two-fold purpose. On the one hand, they will help winnow out alternatives along the global bifurcation curve, and on the other, they will be key in proving both upper and lower bounds on the bifurcation parameter $\ep$. The arguments are similar to those in \cite{hmw:global}. 

\begin{lemma} \label{lem:uniform}
	Suppose $(\phi,\ep)$ solves \eqref{A_formulation} and that \eqref{uniform_bounds_delta} holds for some $\delta>0$. Then  $\|\phi\|_{C^{1+\alpha}}\lesssim_\delta1$.
\end{lemma}
\begin{proof}
	We recall that the problem can be rewritten as the Riemann--Hilbert problem $\Img(A\phi')=0$ where $A$ defined as in \eqref{A} can be rewritten as
	\begin{equation*}
	\ep A=2\partial_z\tilde{\Psi}(\phi(w))w.
	\end{equation*}
	Then \eqref{uniform_bounds_delta} directly yields that
	\begin{equation}\label{bounds_partial_Phi}
	|\ep A|=2|\partial_z\tilde{\Psi}(\phi(w))w|>\delta.
	\end{equation}
	Since the winding number of $\ep A$ is zero by Lemma~\ref{lem:winding_number}, applying Lemma~\ref{lem:RH_problem} yields
	\begin{equation}\label{phi'}
	\phi'(w)=\operatorname{exp}\bigg\{\frac{w}{2\pi}\int_{\T}\frac{1}{\tau-w}\bigg[\frac{1}{\xi}\operatorname{arg}\bigg(\frac{\partial_z\tilde{\Psi}(\phi(\xi))\xi}{\overline{\partial_z\tilde{\Psi}(\phi(\xi))\xi}} \bigg)\bigg]_{\xi=w}^{\xi=\tau}\,d\tau\bigg\}.
	\end{equation}
	Combining Lemmas~\ref{lem:Koebe_1/4} and~\ref{lem:bounds_on_phi'}, a Sobolev embedding yields
	\begin{equation*}
	\|\phi\|_{C^\sigma}<C\qquad\text{for some }\sigma\text{ depending on }p.
	\end{equation*}
	From Lemma~\ref{lem:elliptic_estimates}, for some arbirary but fixed $\beta$ we have $\|\partial_z\tilde{\Psi}\|_{C^\beta(\overline{D})}<C$. Hence we get $\| \partial_z\tilde{\Psi}\circ\phi\|_{C^{\sigma\beta}(\mathbb{T})}<C$. From \eqref{bounds_partial_Phi}, one can now easily show that
	\begin{equation*}
	\bigg\|\frac{1}{\tau}\operatorname{arg}\bigg(\frac{\partial_z\tilde{\Psi}(\phi(\tau))\tau}{\overline{\partial_z\tilde{\Psi}(\phi(\tau))\tau}}\bigg)\bigg\|_{C^{\sigma\beta}}<C.
	\end{equation*}
	Using the fact that the Cauchy integral is a bounded operator $C^{\sigma\beta}\to C^{\sigma\beta}$, composition with the exponential now yields $\|\phi'\|_{C^{\sigma\beta}}<C$. Repeating the above argument with $\sigma=\beta=\sqrt{\alpha}$ yields $\|\phi'\|_{C^\alpha}<C$.
\end{proof}

\begin{lemma}\label{lem:phi_bounds}
	Let $(\phi,\ep)$ is a solution to \eqref{A_formulation} and suppose that \eqref{uniform_bounds_delta} holds for some $\delta>0$. Then 
	\begin{equation*}
	|\phi'|,|\phi|\gtrsim_\delta1.
	\end{equation*}
\end{lemma}
\begin{proof}
	By Lemma~\ref{lem:uniform}, we know that $\|\phi\|_{C^{1+\alpha}}\lesssim_\delta1$. For the lower bound on $|\phi'|$ we simply take the multiplicative inverse of \eqref{phi'} and use the bounds on $\|\phi\|_{C^\alpha}$. By a similar argument as the one in the proof of Lemma~\ref{lem:uniform}, we find that $\|1/\phi'\|_{C^\alpha}\gtrsim_\delta 1$, which in turn implies that $\operatorname{min}_\mathbb{T}|\phi'|\gtrsim_\delta1$.
	
	We now turn to the lower bound on $|\phi|$. To begin with, we easily get a lower bound on $\|\phi\|_{L^\infty}$ using the Schwarz lemma. Indeed, since the function $g(w)=\Phi(w)/w$ is holomorphic at infinity with $g(\infty)=1$, by the modulus maximum principle we get
	\begin{equation}\label{phi_L_infty}
	\|\phi\|_{L^\infty(\mathbb{T})}=\|g\|_{L^\infty(\mathbb{C}\setminus\mathbb{D})}>1.		
	\end{equation}
	Notice that since $\phi$ is continuous, there exists $\theta_1,\theta_2\in[0,2\pi]$ such that
	$$
	\min_{\mathbb T}|\phi|=\phi(e^{i\theta_2}), \quad \max_{\mathbb T}|\phi|=\phi(e^{i\theta_1}).
	$$
	Hence
	\begin{equation*}
	\log\frac{\min_{\mathbb T}|\phi|}{\max_\mathbb T|\phi|}=\log\bigg|\frac{\phi(e^{i\theta_2})}{\phi(e^{i\theta_1})}\bigg|=-\operatorname{Re}\int_{\theta_1}^{\theta_2}\frac{d}{dt}\log\phi(e^{it})\,dt=\int_{\theta_1}^{\theta_2}\Img\frac{e^{it}\phi'(e^{it})}{\phi(e^{it})}\,dt.
	\end{equation*}
	We now estimate this integral to find
	\begin{equation*}
	\log\frac{\min_{\mathbb T}|\phi|}{\max_\mathbb T|\phi|}\leq\left|\int_{\theta_1}^{\theta_2}\frac{|\phi'(e^{it})|}{|\phi(e^{it})|}\,dt\right|\leq\frac{2\pi}{\cos(\pi/2-\delta_1)}.
	\end{equation*}
	Finally, taking exponentials and using \eqref{phi_L_infty}, yields
	\begin{equation*}
	\min_{\mathbb T}|\phi|\geq\max_\mathbb T|\phi|\exp\bigg(-\frac{2\pi}{\cos(\pi/2-\delta_1)} \bigg)>\exp\bigg(-\frac{2\pi}{\cos(\pi/2-\delta_1)} \bigg)\gtrsim_\delta1,
	\end{equation*}
	thus concluding the proof.
\end{proof}

\subsection{Upper bounds on $\ep$}
Using the results from the previous subsection, we find upper bounds on the bifurcation parameter $\ep$, depending only on $\delta$ from \eqref{uniform_bounds_delta} and on $l$, the distance of the center of the patch $D_1$ to the $y$-axis. On the one hand, we also prove a result which will later enable us to show that $\ep$ can only go to $0$ along the local curve of solutions.

\begin{lemma}\label{lem:ep_bounds}
	Let $(\phi,\ep)$ solve \eqref{A_formulation} and suppose that  \eqref{uniform_bounds_delta} holds for some $\delta>0$. Then $\ep\lesssim_\delta l$. 
\end{lemma}
\begin{proof}
	Since the centered renormalized patch $\overline{D}$ is compact, there exists an $w\in\mathbb{T}$ such that
	\begin{equation*}
	\phi(w)=(-x^*,0).
	\end{equation*}
	Since \eqref{uniform_bounds_delta} holds, we know from Lemma~\ref{lem:phi_bounds} that $|\phi|\gtrsim_\delta1$. Hence we have $x^*\geq1/C$. A point on $\partial D$ gets mapped to a point on $\partial D_1$ by $l+\ep\phi$. Hence we have
	\begin{equation*}
	l-\ep(x^*,0)\in\partial D_1.
	\end{equation*}
	Let us now assume that $\ep>Cl$. Then there exists a $w_1\in\partial D_1$ whose $x$-coordinate is negative. However, this contradicts the assumption that $\partial D_1\subset\{(x,y)\in\mathbb{R}^2:\,x>0\}$, necessary assumption to ensure that the patches $D_1$ and $D_2$ don't intersect. Hence, we must have $\ep\leq Cl$.
\end{proof} 

In what follows we define the $L^p(\T)$ norms, $p<\infty$, in the usual way
\begin{equation}\label{end0.5}
\|g\|_{L^p(\T)}^p:=\int_0^{2\pi}|g(e^{it})|^p\,dt=\int_{\T}|g(w)|^p(iw)^{-1}\,dw.
\end{equation}

\begin{lemma}\label{thm:rigidity2}
	Assume that $\omega_0=(\pi\ep^2)^{-1}(\chi_{D_1}+\chi_{D_2})$ is a solution of the system \eqref{elliptic formulation}. Assume that $\delta>0$, $C(\delta)\geq 1$ is a constant that depends only on $\delta$, and the function $\phi$ corresponding to $\omega_0$ (as in  \eqref{slov5}) satisfies the bounds
	\begin{equation}\label{end0}
	\|\phi'\|_{L^\infty(\T)}\leq C(\delta).
	\end{equation}
	Then $\phi$ is of the form $\phi(w)=w+\ep f(w)$ (as in \eqref{slov5}) with
	\begin{equation}\label{end2}
	\|f'\|_{L^2(\T)}\lesssim_\delta 1.
	\end{equation}
\end{lemma}

\begin{proof} We prove first that for any $\kappa>0$ there is $\ep_1=\ep_1(\delta,\kappa)$ such that if $\omega_0$ and $\phi$ are as in the statement of the lemma and $\ep\leq \ep_1$ then 
	\begin{equation}\label{Easyend2}
	\|\phi-w\|_{L^\infty(\T)}\leq\kappa.
	\end{equation}
	We argue by contradiction. Assume that there exists a sequence $\phi_n$ with corresponding $\ep_n$ such that
	\begin{equation}\label{contra}
	|\phi_n(w)-w|_{L^\infty}\geq\kappa,\qquad \lim_{n\to\infty}\ep_n=0.
	\end{equation}
	We have
	\begin{equation*}
	\phi_n(w)=w+\sum_{k\geq 1}a_{n,k}w^{-k}.
	\end{equation*}
	Since the sequence $\phi_n$ is uniformly bounded in $H^1$ (due to \eqref{end0}), we get that $\phi_n$ converges strongly to $\phi$, in any space weaker than $H^1$, in the sense that
	\begin{equation}\label{convergence}
	\lim_{n\to\infty}a_{n,k}=a_k\qquad\text{for every }k\geq1.
	\end{equation}
	Moreover, the uniform $H^1$ bounds imply the rapid decay
	\begin{equation}\label{decay}
	\sum_{k\geq 1}|ka_{n,k}|^2\lesssim_\delta 1\qquad\text{and hence also}\qquad\sum_{k\geq 1}|ka_{k}|^2\lesssim_\delta 1.
	\end{equation}
	We now define
	\begin{equation*}
	\Phi(w)=w+\sum_{k\geq 1}a_kw^{-k} \qquad\text{for }|w|\geq1.
	\end{equation*}
	The map $\Phi$ is holomorphic for $|w|\geq1$ due to the decay in \eqref{decay}. Moreover, from \eqref{convergence}--\eqref{decay} we see that $\Phi=\lim_{n\to\infty}\Phi_n$ uniformly in $\C\setminus \mathbb{D}$. Hence, by the Hurowitz Theorem, $\Phi$ must be a univalent conformal map from $\mathbb{C}\setminus\overline{\mathbb{D}}$ to $\mathbb{C}\setminus\overline{D}$. From \eqref{end0}, it is easy to see that $\phi$ must hence also be univalent and thus be a Jordan curve.
	
	From properties of conformal maps (see for instance Section 1.2 in \cite{pommerenke:book}), we know that the area $|D_n|^2$ converges to $|D|^2$. Moreover, since $D$ is a Jordan domain, we get
	\begin{equation*}
	\lim_{n\to\infty}\int_{D_n}p_n(y)dy=\int_{D}p(y)dy,
	\end{equation*} 
	where here the convergence of $p_n$ follows from classical maximum principle arguments. From \eqref{slov8}, we see that $D$ must be the disk $\mathbb{D}$, a contradiction to \eqref{contra}. This completes the proof of \eqref{Easyend2}.
	
	We turn now to the proof of \eqref{end2}. In view of the assumption \eqref{end0} we may assume that $\ep\leq\ep_0(\delta)$ is very small. We consider the formulation of the problem without singularity in $\ep$, as given in \eqref{formulation-without-epsilon}. We multiply both sides of the equality by $(-4\pi/w)f'(w)$ and integrate to get
	\begin{equation}\label{f_estimate}
	\int_{\T}\frac{1}{2\pi}\textnormal{Im}\left[f'(w)\right]\frac{-4\pi}{w}f'(w)dw+\int_{\T}\textnormal{Im}\left[\left\{\mathscr J(f,\ep)(w)-\Omega(\ep \overline{\phi(w)}+l)\right\}w\phi'(w)\right]\frac{-4\pi}{w}f'(w)dw=0.
	\end{equation}
	Using \eqref{end0.5} and the general form of $f$ in \eqref{slov5} we notice that
	\begin{equation*}
	\int_{\T}\frac{1}{2\pi}\textnormal{Im}\left[f'(w)\right]\frac{-4\pi}{w}f'(w)dw=\int_{\T}\frac{|f'(w)|^2-[f'(w)]^2}{iw}dw=\|f'\|_{L^2(\T)}^2.
	\end{equation*}
	We use the formulas \eqref{J} and \eqref{f_estimate} to conclude that
	\begin{equation}\label{f_estimate2}
	\|f'\|_{L^2(\T)}^2+J_1+J_2+J_3=0,
	\end{equation}
	where
	\begin{equation*}\label{f_estimate3}
	\begin{split}
	J_1&:=\int_{\T}\textnormal{Im}\left[w\phi'(w)\frac{i}{4\pi^2}\int_{\T} \frac{\overline{\phi(w)}-\overline{\phi(\xi)}}{\phi(w)-\phi(\xi)}f'(\xi)d\xi\right]\frac{-4\pi}{w}f'(w)dw,\\
	J_2&:=\int_{\T}\textnormal{Im}\left[w\phi'(w)\frac{-1}{2\pi^2}\int_{\T} \frac{\textnormal{Im}\big[(w-\xi)(\overline{f(w)}-\overline{f(\xi)})\big]}{(w-\xi)(\phi(w)-\phi(\xi))}d\xi\right]\frac{-4\pi}{w}f'(w)dw,\\
	J_3&:=\int_{\T}\textnormal{Im}\left[w\phi'(w)\left\{\frac{-i}{4\pi^2}\int_{\T} \frac{\overline{\phi(\xi)}}{\varepsilon \phi(\xi)+\varepsilon\phi(w)+2l}\phi'(\xi)d\xi-\Omega(\ep \overline{\phi(w)}+l)\right\}\right]\frac{-4\pi}{w}f'(w)dw.
	\end{split}
	\end{equation*}
	We will show that there are constants $C_1(\delta)\geq 1$ and $\ep_0(\delta)>0$ such that if $\ep\leq\ep_0(\delta)$ then
	\begin{equation}\label{estimate_J}
	|J_n|\leq C_1(\delta)\|f'\|_{L^2(\T)}+(1/10)\|f'\|_{L^2(\T)}^2,\qquad n\in\{1,2,3\}.
	\end{equation}
	Assuming these bounds, the desired conclusion \eqref{end2} would follow using \eqref{f_estimate2}.
	
	The bounds \eqref{estimate_J} follow easily if $n=3$ (in the stronger form $|J_3|\lesssim _\delta\|f'\|_{L^2(\T)}$) using the Cauchy inequality, \eqref{end0}, and \eqref{slov8}.  We prove these bounds now for $n=1$. Notice that
	\begin{equation*}
	\int_{\T}\frac{\overline{\phi(w)}-\overline{\phi(\xi)}}{w-\xi}f'(\xi)d\xi=0\qquad\text{ for any }w\in\T,
	\end{equation*}
	due to the form of the functions $f,\phi$ in \eqref{slov5}. Let $g(w):=\phi(w)-w=\ep f(w)$. In view of \eqref{Easyend2} we may assume that $\|g\|_{L^\infty(\T)}\leq \kappa$, for some sufficiently small constant $\kappa=\kappa(\delta)>0$. Then we estimate
	\begin{equation*}\label{step_0}
	\begin{split}
	\bigg|\int_{\T}\frac{\overline{\phi(w)}-\overline{\phi(\xi)}}{w+g(w)-\xi-g(\xi)}f'(\xi)d\xi\bigg|&=\bigg|\int_{\T}\frac{(\overline{\phi(w)}-\overline{\phi(\xi)})(g(w)-g(\xi))}{(w-\xi+g(w)-g(\xi))(w-\xi)}f'(\xi)d\xi\bigg|\\
	&\leq\int_{0}^{2\pi}\frac{\big|\phi(e^{it})-\phi(e^{is})\big|}{\big|(e^{it}-e^{is}+g(e^{it})-g(e^{is})\big|}\frac{\big|g(e^{it})-g(e^{is})\big|}{\big|e^{it}-e^{is}\big|}|f'(e^{is})|ds\\
	&\leq\int_0^{2\pi}\frac{\big|g(e^{it})-g(e^{is})\big|}{\big|e^{it}-e^{is}\big|}|f'(e^{is})|ds\\
	&\lesssim(\kappa C(\delta))^{1/2}\int_0^{2\pi}|e^{it}-e^{is}|^{-1/2}|f'(e^{is})|ds,
	\end{split}
	\end{equation*}
	for any $w=e^{it}\in\T$. The last inequality follows by estimating $|g(e^{it})-g(e^{is})|\lesssim (\kappa C(\delta))^{1/2}|e^{it}-e^{is}|^{1/2}$ (since $\|g'\|_{L^\infty(\T)}\lesssim C(\delta)$) and $|g(e^{it})-g(e^{is})|\lesssim\kappa$ (due to \eqref{Easyend2}). Taking the geometric mean and dividing by $|e^{it}-e^{is}|$ yields the claimed bound. Therefore
	\begin{equation*}
	\bigg\|\int_{\T}\frac{\overline{\phi(w)}-\overline{\phi(\xi)}}{w+g(w)-\xi-g(\xi)}f'(\xi)d\xi\bigg\|_{L^2_w(\T)}\lesssim(\kappa C(\delta))^{1/2}\|f'\|_{L^2(\T)}.
	\end{equation*}
	The desired bounds \eqref{estimate_J} for $n=1$ follow if $\kappa(\delta)$ is chosen sufficiently small (depending only on the constant  $C(\delta)$ in \eqref{end0}). 
	
	Finally, we bound the integral $|J_2|$. The main difficulty has to do with the factor $|\phi(w)-\phi(\xi)|$ in the denominator, which could be very small. Notice first that
	\begin{equation*}
	\int_{\T}\frac{\Img\big[(w-\xi)(\overline{f(w)}-\overline{f(\xi)}) \big]}{(w-\xi)(w-\xi)}d\xi=0
	\end{equation*}
	for any $w\in\T$, due to the form of the function $f$ in \eqref{slov5} and the residue theorem. Therefore
	\begin{equation*}
	\int_{\T}\frac{\Img\big[(w-\xi)(\overline{f(w)}-\overline{f(\xi)}) \big]}{(w-\xi)(\phi(w)-\phi(\xi))}d\xi=-\int_{\T}\frac{\Img\big[(w-\xi)(\overline{f(w)}-\overline{f(\xi)}) \big]}{(w-\xi)^2(\phi(w)-\phi(\xi))}\big(g(w)-g(\xi)\big)d\xi=:F_1(w)+F_2(w),
	\end{equation*} 
	where
	\begin{equation*}
	\begin{split}
	&F_1(w):=-\int_{|w-\xi|\leq \rho}\frac{\Img\big[(w-\xi)(\overline{f(w)}-\overline{f(\xi)}) \big]}{(w-\xi)^2(\phi(w)-\phi(\xi))}\big(g(w)-g(\xi)\big)d\xi,\\
	&F_2(w):=-\int_{|w-\xi|\geq \rho}\frac{\Img\big[(w-\xi)(\overline{f(w)}-\overline{f(\xi)}) \big]}{(w-\xi)^2(\phi(w)-\phi(\xi))}\big(g(w)-g(\xi)\big)d\xi,
	\end{split}
	\end{equation*}
	where $\rho=\rho(\delta)>0$ is a small constant to be fixed. Then we estimate, for any $w=e^{it}\in\T$,
	\begin{equation}\label{2_11}
	\begin{split}
	|F_1(w)|&=\bigg|\int_{|w-\xi|\leq\rho}\frac{\Img\big[(w-\xi)(\overline{f(w)}-\overline{f(\xi)}) \big]}{(w-\xi)^2(\phi(w)-\phi(\xi))}\ep\big(f(w)-f(\xi)\big)d\xi\bigg|\\
	&=\bigg|\int_{|w-\xi|\leq \rho}\frac{\Img\big[(w-\xi)(\overline{\phi(w)}-\overline{\phi(\xi)}-(\overline{w}-\overline{\xi})) \big]}{(w-\xi)^2(\phi(w)-\phi(\xi))}\big(f(w)-f(\xi)\big)d\xi\bigg|\\
	&\lesssim\int_{s\in[0,2\pi],\,|e^{it}-e^{is}|\leq\rho}\bigg|\frac{f(e^{it})-f(e^{is})}{e^{it}-e^{is}}\bigg|ds\\
	&\lesssim \rho^{1/2}\|f'\|_{L^2(\T)},
	\end{split}
	\end{equation} 
	where the estimate in the last line holds because $|f(e^{it})-f(e^{is})|\lesssim \|f'\|_{L^2(\T)}|e^{it}-e^{is}|^{1/2}$. Moreover, if $|w-\xi|\geq\rho$ and $\|g\|_{L^\infty(\T)}\leq\rho^2/10$ is sufficiently small then $|\phi(w)-\phi(\xi)|\geq |w-\xi|-|g(w)-g(\xi)|\geq |w-\xi|-\rho/5\geq |w-\xi|/2$, and we estimate
	\begin{equation}\label{2_17}
	\begin{split}
	|F_2(w)|&\lesssim\int_{s\in[0,2\pi],\,|e^{it}-e^{is}|\geq\rho}\frac{\big|(e^{it}-e^{is})(\overline{f(e^{it})}-\overline{f(e^{is})}) \big|}{\big|(e^{it}-e^{is})^2(\phi(e^{it})-\phi(e^{is}))\big|}\big|g(e^{it})-g(e^{is})\big|ds\\
	&\lesssim \rho\|f\|_{L^\infty(\T)}.
	\end{split}
	\end{equation}  
	The desired bounds \eqref{estimate_J} follow for $n=2$ from \eqref{2_11}--\eqref{2_17}, by taking $\rho$ sufficiently small.
\end{proof}

\section{Functional Analytic setting}
\subsection{Function spaces}
We will work in the Banach spaces 
\begin{align*}
&\mathcal{X}^{k+\alpha}=\bigg\{f\in C^{k+\alpha}(\mathbb{T})\colon f(w)=\sum_{n\geq 1}a_nw^{-n},\;a_n\in\mathbb{R}\bigg\},\\&
\mathcal{Y}^{k-1+\alpha}=\bigg\{f\in C^{k-1+\alpha}(\mathbb{T})\colon f(e^{i\theta})=\sum_{n\geq 2}a_n\sin(n\theta),\;a_n\in\mathbb{R}\bigg\}.
\end{align*}
We recall that $\ep f(w)=\phi(w)-w$.  For convenience we will work  in the open subspace
\begin{align*}
\mathcal{U}^{k+\alpha}:=\mathcal{U}_1^{k+\alpha}\cap\mathcal{U}_2^{k+\alpha}\cap\mathcal{U}_3^{k+\alpha},
\end{align*}
where
\begin{align*}
&\mathcal{U}_1^{k+\alpha}=\bigg\{(f,\ep)\in\mathcal{X}^{k+\alpha}\times\mathbb{R}:|\ep A|>0 \bigg\},\\
&\mathcal{U}_2^{k+\alpha}=\bigg\{(f,\ep)\in\mathcal{X}^{k+\alpha}\times\mathbb{R}:\operatorname{Re}\bigg(\frac{w\phi'}{\phi}\bigg)>0 \bigg\},\\
&\mathcal{U}_3^{k+\alpha}=\Big\{(f,\ep)\in\mathcal{X}^{k+\alpha}\times\mathbb{R}:\inf_{\tau\neq w}|\ep\phi(\tau)+\ep\phi(w)+2l|>0 \Big\}.
\end{align*}
The definition of the open subspaces comes from the requirement that \eqref{uniform_bounds_delta} must hold. We remark that, as we have seen in Section~\ref{sect:formulation}, the singularity at $\ep=0$ is removable and hence the definition of the space $\mathcal{U}_1$ does not prevent $\ep$ from going to $0$. Specifically, the point $(f,\ep)=(0,0)\in\mathcal{U}_1$.

\subsection{Operator formulation and solution set}
We define the nonlinear operator 
\begin{equation*}
\tilde{\mathscr F}^{k+\alpha}\colon \mathcal{U}^{k+\alpha}\times \mathbb{R}\to\mathcal{Y}^{k-1+\alpha},
\end{equation*}
where here
\begin{equation*}\label{nonlinear-op}
\tilde{\mathscr F}^{k+\alpha}(f,\ep,\Omega):=\Img\bigg\{\bigg(\Omega\big((\overline w+\ep \overline{f})+l\big)+\tfrac{1}{\pi\ep}\mathcal{C}(w+\ep f)(\overline{w}+\ep\overline{f})-\tfrac{1}{\pi \ep}\tilde{\mathcal{C}}_{\ep,l}(w+\ep f)(\overline{w}+\ep\overline{f}) \bigg)w(1+\ep f') \bigg\}.
\end{equation*}

Let us remark that the functions in the range space $\mathcal{Y}^{k-1+\alpha}$ have vanishing first Fourier coefficient $a_1$. This technical point is due to the invertibility of the linear operator at the trivial solution in order to obtain the local curve of solutions. However, the nonlinear operator is not well-defined in those spaces. Nevertheless, we can work with the extra parameter $\Omega$ in order to get that the problem is well-posed. That is, that the first Fourier coefficient of $\tilde{\mathscr{F}}$ vanishes. As stated in \cite{hm:pairs, garcia:choreography} we can define the dependence of $\Omega$ on $f$ and $\ep$ in the following way:
\begin{equation}\label{omega}
\Omega(f,\ep)=\frac{\int_{\T} \mathscr{J}(f,\ep)(w)(w-\overline{w})(1+\ep f'(w))dw}{\int_{\T} (1+\ep f'(w))(w-\overline{w})(l+\ep \overline{w}+\ep^2 f(\overline{w})dw}.
\end{equation}
Moreover, in \cite{hm:pairs, garcia:choreography} we observe the following properties of $\Omega$.
\begin{proposition}\label{prop-omega}
	For any $\alpha\in(0,1)$ we have that $\Omega:\mathcal{U}^{k+\alpha}\times\R\rightarrow \R$ is well-defined. Furthermore,
	\begin{itemize}
		\item[i)] $\Omega(f,0)=\Omega_0$, for any $f$.
		\item[ii)] The first Fourier coefficient of $\tilde{\mathscr{F}}(f,\ep,\Omega)$ vanishes if and only if $\Omega=\Omega(f,\ep)$.
	\end{itemize}
\end{proposition}

Hence, from this point onwards, we will work with ${\mathscr{F}}^{k+\alpha}$  defined as
$$
\mathscr{F}^{k+\alpha}(f,\ep):=\tilde{\mathscr{F}}^{k+\alpha}(f,\ep,\Omega(f,\ep)),
$$
where we take into account that $\Omega$ is fixed and depends on $(f,\ep)$.

We will often rewrite this operator as
\begin{equation}\label{nonlinear_operator}
\mathscr F^{k+\alpha}(f,\ep)=\Img\big\{\mathscr A(f,\ep)(1+\ep f') \big\},
\end{equation}
where $\ep \mathscr A (f,\ep)=\ep A(\phi,\Omega(f,\ep),\ep)$ (see \eqref{A}) is analytic $\mathcal{U}_1^{k+\alpha}\to C^{k+\alpha}(\mathbb{T})$ for any $k\geq1$ and $\alpha\in(0,1)$. Moreover, the linearized operator is then given by
\begin{equation*}\label{linearized_operator}
\mathscr F_f^{k+\alpha}(f,\ep)g=\Img\big\{ \mathscr A(f,\ep)g'\big\}+\Img\big\{ (1+\ep f')\mathscr A_f(f,\ep)g\big\}.
\end{equation*}
Finally, for convenience of notation throughout the paper, following \cite{hmw:global}, we define the set of solutions as
\begin{equation*}
\mathscr S^{k+\alpha}:=\{(f,\ep)\in\mathcal{U}^{k+\alpha}:\mathscr F^{k+\alpha}(f,\ep)=0\}.
\end{equation*}
Throughout this entire section, $k\geq1$ and $\alpha\in(0,1)$ have been arbitrary. In the rest of the paper, we will for the most part fix $\alpha$ and fix $k=3$. For this case, for simplicity of notation, we denote 
\begin{align*}
\mathcal{X}:=\mathcal{X}^{3+\alpha},\qquad\mathcal{Y}:=\mathcal{Y}^{2+\alpha},\qquad\mathcal{U}:=\mathcal{U}^{3+\alpha},\qquad\mathcal{U}_i:=\mathcal{U}_i^{3+\alpha}\quad\text{for }i=1,2,\cdots5.
\end{align*} 
Moreover, we define
\begin{equation*}
\mathscr F:=\mathscr F^{3+\alpha}\qquad\text{and}\qquad\mathscr S:=\mathscr S^{3+\alpha}.
\end{equation*}

Finally, from Remark~\ref{rem:analytic} and thanks to $\mathcal{U}_1$ (which prevents any singularities in $\ep$) we obtain the fact that the operator $\mathscr F$ is real-analytic.

\begin{proposition}\label{prop-welldef}
	$\mathscr F:\mathcal{U}\rightarrow \mathcal{Y}$ is well-defined and real-analytic.
\end{proposition}

Although of lesser importance for the construction of the local curve, Proposition \ref{prop-welldef} will be key for the global continuation. Indeed, the analyticity of our operator gives us access to powerful global bifurcation theorems.

\section{Existence Results}
We now have all the necessary tools to construct our local and global curve of solutions.

\subsection{Local curve of solutions}\label{sect:local}

This section aims to study the existence of the local curve. Although this construction was already carried out in \cite{hm:pairs} (see also \cite{garcia:choreography}), for the benefit of the reader, we recall all the details here.

At the level of the local curve, we have to check properties of the linearized operator at $\varepsilon=0$. To this end, it is advantageous to use the formulation without the singularity in $\varepsilon$, as given in \eqref{formulation-without-epsilon}. As a result, we can reformulate $\mathscr{F}$ as
\begin{align}
\mathscr F(f,\ep)(w)=&\frac{1}{2\pi}\textnormal{Im}\left[f'(w)\right]+\textnormal{Im}\left[\left\{\mathscr J(f,\ep)(w)-\Omega(f,\ep)(\overline{{\ep}\phi(w)}+l)\right\}w\phi'(w)\right],\label{F-without-epsilon}
\end{align}
where we recall that $\phi(w)=w+\varepsilon f(w)$, $\Omega$ is defined in \eqref{omega} and $\mathscr J$ is defined in \eqref{J}. As remarked previously, the definition of $\mathcal{U}$ ensures that the denominator in $\mathscr J$ never vanishes.

Physically, the bifurcation parameter $\ep$ we use denotes the conformal radius of the first patch, and by symmetry, that of the second one as well. Since we wish to bifurcate from point vortices, we begin by checking that $\ep=0$ does indeed yield a trivial solution to our problem.  
\begin{proposition}[Trivial solution]\label{pro-trivial}
	$\mathscr F\left(0,0\right)(w)=0$, for any $w\in\T$.
\end{proposition}
\begin{proof}
	Recall from Proposition \ref{prop-omega} that $\Omega(f,0)=\Omega_0$. From \eqref{F-without-epsilon} we find that
	\begin{align*}
	\mathscr F(0,0)(w)=\textnormal{Im}\left[w\mathscr J(0,0)(w)-\Omega_0 l w\right].
	\end{align*}
	Let us compute $\mathscr J(0,0)$ by using the expression given in \eqref{J}:
	\begin{align*}
	\mathscr J(0,0)(w)=&-\frac{i}{4\pi^2}\frac{1}{2l}\int_{\T} \frac{d\xi}{\xi}=\frac{1}{4\pi l}.
	\end{align*}
	Hence
	$$
	\mathscr F(0,0)(w)=\textnormal{Im}\left[w\left\{\frac{1}{4\pi l}-\Omega_0 l \right\}\right],
	$$
	which from the expression of $\Omega_0$ in \eqref{Omega0} implies that $\mathscr F\left(0,0\right)(w)=0$, for any $w\in\T$.
\end{proof}

The main tool we will use in order to construct a local curve of solutions is the analytic implicit function theorem recalled in Appendix~\ref{app:bif_thms}. We thus begin by computing the Fréchet derivative of the operator $\mathscr F$, as defined in \eqref{F-without-epsilon}, at the trivial solution $(0,0)$.

\begin{proposition}[Expression of $\partial_{f} \mathscr F(0,0)$.] \label{prop-dfF}
	$$
	\partial_f \mathscr F(0,0)=\frac{1}{2\pi}\textnormal{Im}\left[h'(w)\right].
	$$
\end{proposition}
\begin{proof}
	From Proposition \ref{prop-omega} we have that $\Omega(f,0)=\Omega_0$ for any $f$, which implies that $\partial_f \Omega(0,0)=0$. Moreover, we calculate
	
	\begin{align*}
	\partial_f \mathscr F(0,0)h(w)&=\frac{1}{2\pi}\textnormal{Im}\left[h'(w)\right]+\lim_{\varepsilon\rightarrow 0}\, \varepsilon \textnormal{Im}\left[\left\{\mathscr J(0,0)(w)-\Omega_0 l\right\}wh'(w)\right]\\
	&+\textnormal{Im}\left[w\partial_f \mathscr J(0,0)h(w)\right]\\
	&=\frac{1}{2\pi}\textnormal{Im}\left[h'(w)\right]+\textnormal{Im}\left[w\partial_f \mathscr J(0,0)h(w)\right].
	\end{align*}
	Computing $\partial_f \mathscr J(0,0)$ from \eqref{J} yields
	\begin{align*}
	\partial_f \mathscr J(0,0)h(w)=\frac{i}{4\pi^2}\int_{\T}\frac{\overline{w-\xi}}{w-\xi}h'(\xi)d\xi-\frac{1}{2\pi^2}\int_{\T} \frac{\textnormal{Im}\left[(w-\xi)(h(\overline{w})-h(\overline{\xi}))\right]}{(w-\xi)^2}d\xi=0,
	\end{align*}
	where, as in \cite{hm:pairs,garcia:choreography}, the last integrals vanish by the residue theorem.
\end{proof}

We are now ready to construct our local curve. The main idea is to \textit{desingularize} the point vortices by perturbative methods. That is, each point along the local curve denotes patches of infintesimally small radii $\ep>0$.

\begin{theorem}[Local curve of solutions]\label{thm:local}
	There exists $\varepsilon_0>0$ and an analytic function $f(\varepsilon)$ such that $\mathscr F(f(\ep),\ep)=0$, for any $\varepsilon\in[0,\varepsilon_0)$. Moreover, the following properties hold: 
	\begin{enumerate}[label=\rm(\roman*)]
		\item \label{uniqueness} \textup{(Uniqueness)} if $(f,\ep)\in\mathcal{X}$ are sufficiently small, then $\mathscr{F}(f,\ep)=0$ implies $f=f(\ep)$;
		
		\item \label{invertibility} \textup{(Invertibility)} for all $0\leq\ep<\ep_0$, the linearized operator $\mathscr{F}_f(f(\ep),\ep):\mathcal X \rightarrow \mathcal Y$ is invertible. 
	\end{enumerate}
	Finally, we denote by 
	\begin{equation*}
	\mathscr{C}_\textup{loc}=\{(f(s),\ep(s)):0\leq s\ll1 \}\subset\mathcal{X}
	\end{equation*}
	the local continuous one-parameter curve of nontrivial solutions to $\mathscr{F}(f,\ep)=0$ for $\ep\in[0,\ep_0)$.
\end{theorem}

\begin{proof}
	From Proposition \ref{prop-welldef}, we have that $\mathscr F:\mathcal{U}\rightarrow  \mathcal{Y}$ is analytic.
	Moreover $\mathscr F(0,0)(w)=0$, for any $w\in\T$ via Proposition \ref{pro-trivial}. We now apply the analytic version of the implicit function theorem (see Theorem~\ref{thm:implicit_function} in Appendix~\ref{app:bif_thms}) to $\mathscr F$ (see \cite{kielhofer}).
	
	From the expression of $\partial_f \mathscr F(0,0)$ we observed that this Fréchet derivative is clearly an isomorphism from $\mathcal X$ to $\mathcal Y$. Indeed, this follows from the fact that we have set the condition in the range space that the first Fourier coefficient must vanish. The uniqueness \ref{uniqueness} follows directly from the implicit function theorem. 
	
	It remains to show \ref{invertibility}. Since $\partial_f \mathscr F(0,0)$ is an isomorphism, due to the continuity of $\partial_{f}\mathscr F$,  it follows that for $\varepsilon$ small $\partial_f \mathscr F(f,\ep)$ also is an isomorphism (see for example, \cite[Lemma 2.5.1]{bt:analytic}).
\end{proof}

Finally, we conclude this section by remarking that when $\ep=0$ we must necessarily have the trivial solution $f=0$. 
\begin{lemma}\label{lem:epsilon0}
	For any $(f,\ep)\in\mathscr S$, if $\ep=0$ then $f=0$.
\end{lemma}
\begin{proof}
	We will compute the expression \eqref{F-without-epsilon} for $\ep=0$. We will use the fact that $\Omega(0,f)=\Omega_0$ for any $f$ (see Proposition \ref{prop-omega}), where $\Omega_0$ is the angular velocity associated to the point vortices defined in \eqref{Omega0}. We get
	\begin{align*}
	\mathscr{F}(0,f)=\frac{1}{2\pi}\Img[f'(w)]+\Img\left[\left\{\mathscr{J}(0,f)-\Omega_0 l\right\}w\right],
	\end{align*}
	where
	\begin{align*}
	\mathscr{J}(0,f)=&\frac{i}{4\pi^2}\int_{\T} \frac{\overline{w-\xi}}{w-\xi}f'(\xi)d\xi-\frac{1}{2\pi^2}\int_{\T} \frac{\textnormal{Im}\left[(w-\xi)(f(\overline{w})-f(\overline{\xi}))\right]}{(w-\xi)^2}d\xi-\frac{i}{4\pi^2}\frac{1}{2l}\int_{\T} \frac{1}{\xi}d\xi.
	\end{align*}
	The last integral can be easily solved with the residue theorem as
	$$
	\int_{\T} \frac{1}{\xi}d\xi=2\pi i.
	$$
	Moreover, since $f\in \mathcal{X}$, the first two integral vanish via the residue theorem, (see the proof of Proposition \ref{prop-dfF}). As a result, we have
	\begin{align*}
	\mathscr{J}(0,f)=&\frac{1}{4\pi l},
	\end{align*}
	implying
	\begin{align*}
	\mathscr{F}(0,f)=\frac{1}{2\pi}\Img[f'(w)]+\Img\left[\left\{\frac{1}{4\pi l}-\Omega_0 l\right\}w\right]=\frac{1}{2\pi}\Img[f'(w)].
	\end{align*}
	Hence, if $(0,f)$ is a solution of \eqref{F-without-epsilon}, meaning that
	$$
	\mathscr{F}(0,f)(w)=\frac{1}{2\pi}\Img[f'(w)]=0,
	$$
	we get that $f'=0$ and consequently, due to the definition of $\mathcal{X}$, that $f=0$. 
\end{proof}

\subsection{Global continuation}\label{sect:global}
In this section, we will extend the local curve from Section~\ref{sect:local} to a global curve by means of analytic global bifurcation theory. Since our operator formulation \eqref{nonlinear_operator} is not defined at the point $(0,0)$, we cannot directly apply the global bifurcation theorem due to Dancer \cite{dancer:global}  and Buffoni--Toland \cite{bt:analytic}. Indeed, we must begin our continuation argument at an arbitrary point along the local curve, away from the trivial solution. As a result, we must slightly modify the standard analytic global bifurcation theorem to better suit our problem. We provide this result as well as an outline of the proof in Appendix~\ref{app:bif_thms}.

Contrary to solutions along the local curve, those along the global one are more than just mere perturbations of the trivial solution. We must therefore also take into account the nonlinearity of the problem and specifically, study some of its topological properties. In other words, we begin by verifying the linearized operator is Fredholm of index $0$ and that bounded and closed subsets of solutions are compact. The ideas are similar to those in \cite{hmw:global}.
\begin{lemma}[Fredholmness]\label{lem:fredholm}
	For any $(f,\ep)\in\mathscr S$ the linearized operator $\mathscr F_f(f,\ep)\colon\mathcal{X}\to\mathcal{Y}$ is Fredholm with index $0$.
\end{lemma}
\begin{proof}
	The main idea of the proof is to show that the linearized operator
	\begin{equation}\label{linearized_operator2}
	\mathscr F_f^{k+\alpha}(f,\ep)g=\Img\big\{ \mathscr A(f,\ep)g'\big\}+\Img\big\{ (1+\ep f')\mathscr A_f(f,\ep)g\big\},
	\end{equation} 
	consists of the sum of an invertible operator and a compact one. From Lemmas~\ref{lem:winding_number} and~\ref{lem:RH_problem}, the first operator on the right hand side of \eqref{linearized_operator2} is clearly invertible $\mathcal{X}\to\mathcal{Y}$. It remains to show that the second operator is compact $\mathcal{X}^{3+\alpha}\to\mathcal{Y}^{2+\alpha}$. To this end, we choose a bounded sequence $g_n\in\mathcal{X}^{3+\alpha}$, and extract a subsequence so that $g_n\to g$ in $\mathcal{X}^{3+\alpha/2}$. Since $\mathscr A\colon\mathcal{U}^{3+\alpha/2}\to C^{3+\alpha/2}(\mathbb{T})$ is analytic, we get
	\begin{equation*}
	\mathscr A_f(f,\ep)g_n\to\mathscr A_f(f,\ep)g\quad\text{in }C^{3+\alpha/2}(\mathbb{T}),
	\end{equation*} 
	and hence also in $\mathcal{Y}=\mathcal{Y}^{2+\alpha}$, thus concluding the proof.
\end{proof}

In order to show compactness, we first construct a suitable family of closed and bounded sets. To this end, we introduce the set $\mathcal{E}_\delta^{k+\alpha}$ defined by the inequalities
\begin{align}\label{inequalities}
&\min_\mathbb{T}\ep A(\phi,\Omega,\ep)\geq\delta, \qquad \frac{1}{1+|\ep|+\|\phi\|_{C^{1+\beta}}}\geq\delta,\qquad \frac{\pi}{2}-\max_\mathbb{T}\bigg|\operatorname{arg}\frac{w\phi'}{\phi}\bigg|\geq\delta,\\&\min_\mathbb T|\phi'|\geq\delta,\qquad \min_\mathbb T|\phi|\geq\delta,\qquad\min_\mathbb{T}|\ep\phi(\tau)+\ep\phi(w)+2l|\geq\delta\nonumber.
\end{align}

\begin{lemma}\label{lem:E}
	For any $\delta>0$, the set $\mathcal{E}_\delta^{k+\alpha}\subset\mathcal{X}^{k+\alpha}\times\mathbb{R}$ defined by the inequalities \eqref{inequalities} is a closed and bounded subset of $\mathcal{U}^{k+\alpha}$. Moreover, for any $(f,\ep)\in\mathcal{U}^{k+\alpha}$ there exists $\delta>0$ so that $(f,\ep)\in \mathcal{E}_\delta^{k+\alpha}$.
\end{lemma}
\begin{proof}
	We begin by showing that the sets $\mathcal{E}_\delta^{k+\alpha}$ exhaust $\mathcal{U}^{k+\beta}$ as $\delta\to0$. For any $(f,\ep)\in\mathcal{U}^{k+\alpha}$, the bounds on the first and the sixth term follow from $(f,\ep)\in\mathcal{U}_i^{k+\alpha}$ for $i=1$ and $3$ respectively. The second bound is immediate and the remaining ones follow from $(f,\ep)\in\mathcal{U}_2^{k+\alpha}$.
	
	Clearly, $\mathcal{E}_\delta^{k+\alpha}$ is a bounded set for any $\delta>0$. Moreover, $\mathcal{E}_\delta^{k+\alpha}\subset\mathcal{U}^{k+\alpha}$. Indeed, we have $\mathcal{E}_\delta^{k+\alpha}\subset\mathcal{U}_1^{k+\alpha}\cap\mathcal{U}_2^{k+\alpha}\cap\mathcal{U}_3^{k+\alpha}$.
	
	It remains to show that $\mathcal{E}_\delta^{k+\alpha}$ is a closed subset of $\mathcal{U}^{k+\alpha}$. The second, fourth, fifth, and sixth conditions of \eqref{inequalities} are clearly closed, as well as the third in combination with the fourth and the fifth. The third, fourth and fifth inequalities in \eqref{inequalities} imply that $\operatorname{dist}_{\mathcal{X}^{k+\alpha}}\big(\mathcal{E}_\delta^{k+\alpha},\partial\mathcal{U}_2^{k+\alpha} \big)>0.$ Hence, the closure of $\mathcal{E}_\delta^{k+\alpha}$ is contained in $\mathcal{U}^{k+\alpha}$. Moreover, since the mapping $(f,\ep)\to \ep A$ is analytic $\mathcal{U}_1^{k+\alpha}\cap\mathcal{U}_3^{k+\alpha}\to C^{k+\alpha}(\mathbb{T})$, the mapping $\min_\mathbb{T}(f,\ep)\to|\ep A|$ is continuous and so the first condition in \eqref{inequalities} is closed.
\end{proof}

\begin{lemma}[Compactness]\label{lem:compactness}
	For any $\delta>0$, the set $\mathscr S^{1+\alpha}\cap \mathcal{E}_\delta^{1+\alpha}$ is a compact subset of $C^{k}(\mathbb{T})\times\mathbb{R}$. Moreover, there exists a constant $C(k,\delta)>0$ so that any solution $(f,\ep)\in\mathcal{E}_\delta^{1+\alpha}\cap\mathscr S^{1+\alpha}$ satisfies $\|f\|_{C^k(\mathbb{T})}<C.$
\end{lemma}
\begin{proof}
	Recall that our problem can be viewed as the Riemann--Hilbert problem $\Img(A\phi')=0$. From Lemmas~\ref{lem:winding_number} and~\ref{lem:RH_problem}, this problem can be solved explicitly to get
	\begin{align*}
	\ep f'(w)&=\operatorname{exp}\bigg\{\frac{w}{2\pi}\int_{\T}\frac{1}{\tau-w}\bigg[\frac{1}{\tau}\operatorname{arg}\bigg(\frac{\ep A(\tau)}{\overline{\ep A(\tau)}}\bigg)-\frac{1}{w}\operatorname{arg}\bigg(\frac{\ep A(\tau)}{\overline{\ep A(\tau)}}\bigg) \bigg]d\tau \bigg\}-1\\&:=\mathscr G(f,\ep).
	\end{align*}
	
	We first show that $\mathscr G(f,\ep):\mathscr{S}^{k+\alpha}\to C^{k+\alpha}(\mathbb{T})$ is a continuous mapping. From the definition of $\mathcal{U}^{k+\alpha}$, clearly both $\ep A$ and $\ep A/\overline{\ep A}$ are continuous mappings $\mathcal{U}^{k+\alpha}\to C^{k+\alpha}$. The same is true for $\operatorname{arg}(\ep A/\overline{\ep A})$ since the winding number of $\ep A$ is $0$ by Lemma~\ref{lem:winding_number}. Finally, the Cauchy integral operator is a bounded linear operator from $C^{k+\alpha}(\mathbb{T})$ to itself. Composing with the exponential yields the desired result.
	
	This continuity result combined with an iteration argument then implies that
	\begin{equation}\label{continous inclusion}
	\mathscr S^{1+\alpha}\hookrightarrow C^{k+\alpha}\quad\text{is a continuous inclusion for all }k\in\mathbb{N}.
	\end{equation}
	
	We are now ready to show the compactness. Indeed, from Lemma~\ref{lem:E}, we know that $\mathcal{E}_\delta^{1+\alpha}$ is a bounded and closed subset of $\mathcal{U}^{1+\alpha}$ and hence is compact in $\mathcal{U}^{1+\alpha/2}$. Since $\mathscr S^{1+\alpha/2}\subset\mathcal{U}^{1+\alpha/2}$ is closed, we must have that $\mathcal{E}_\delta^{1+\alpha}\cap\mathscr S^{1+\alpha/2}$ is also a compact subset of $\mathscr S^{1+\alpha/2}$. The validity of \eqref{continous inclusion} proves the desired compactness. In particular, $\mathcal{E}_\delta^{1+\alpha}\cap\mathscr S^{1+\alpha/2}$ is bounded in $C^k(\mathbb{T})\times\mathbb{R}$ thus concluding the proof.
\end{proof}

From Lemmas~\ref{lem:fredholm} and~\ref{lem:compactness} along with Theorem \ref{thm: global bifurcation} and the results from Section~\ref{sect:local}, we obtain the following intermediate result.
\begin{theorem}\label{thm: global bifurcation_intermediate}
	There exists a continuous curve $\mathscr C$ of solutions that extends $\mathscr C_{\text{loc}}$, parameterized by $s\in(0,\infty)$, such that
	\begin{enumerate}[label=\rm(\alph*)]
		\item as $s\to\infty$
		\begin{equation*}\label{blow-up}
		\min\bigg\{\min_{w\in\partial D}\partial_r\tilde{\Psi},\;\min_{\T}|\ep\phi(\tau)+\ep\phi(w)+2l| \bigg\}\to 0;
		\end{equation*}
		where here $\phi(s)=w+\ep f(s)$;
		\item Near each point $(\ep(s_0),f(s_0))\in\mathscr C$, we can reparameterize $\mathscr C$ so that $s\mapsto (\ep(s),f(s))$ is real analytic;
		\item \label{thm:global_local} $(\ep(s),f(s))\notin\mathscr C_{\text{loc}}$ for $s$ sufficiently large.
	\end{enumerate} 
\end{theorem}
\begin{proof}
	We apply Theorem~\ref{thm: global bifurcation}. The existence of the local curve and the invertibility of the linearized operator along that curve is guaranteed by Theorem~\ref{thm:local}. By Lemma~\ref{lem:fredholm}, $\mathscr F$ satisfies the Fredholm index $0$ assumption, and by choosing $Q_j=\mathcal{E}_{1/j}^{3+\alpha}$, the compactness assumption is satisfied by Lemmas~\ref{lem:E} and~\ref{lem:compactness}. Moreover, with this choice of $Q_j$, the blow-up scenario in Theorem~\ref{thm: global bifurcation} becomes
	\begin{align*}
	\min\bigg\{	\min_\mathbb{T}\ep|A|,  \frac{1}{1+|\ep|+\|\phi\|_{C^{1+\alpha}}}&, \frac{\pi}{2}-\max_\mathbb{T}\bigg|\operatorname{arg}\frac{w\phi'}{\phi}\bigg|,\,\min_\mathbb T|\phi'|,\, \min_\mathbb T|\phi|,\,\min_{\T}|\ep\phi(\tau)+\ep\phi(w)+2l| \bigg\}\to 0.
	\end{align*}
	Note that the Hölder exponent has been reduced from $3+\alpha$ to $1+\alpha$ due to Lemma~\ref{lem:compactness}. By Lemmas~\ref{lem:Psi_r}--\ref{lem:ep_bounds}  we obtain \eqref{blow-up}.
\end{proof}

The following result tells us that the boundary of the patch $D$ is analytic. The proof follows a proof of Kinderlehrer, Nirenberg, and Spruck \cite{kns:freereg} for elliptic free boundary problems, which was already used in \cite{Hassainia-Hmidi-Masmoudi:2021} for the analyticity of a rotating single patch.

\begin{theorem}\label{thm:analytic}
	Assume that $D\in C^1$ and that $(\tilde{\Psi},\Omega,\ep)$ solves \eqref{elliptic formulation-centered}. If $\tilde{\Psi}\in C^2(\C\setminus D)\cup C^2(\overline{D})$, then $\partial D$ is analytic.
\end{theorem}
\begin{proof}
	The vorticity of the renormalized centered patch $D$ is $\frac{1}{\pi}$, and hence Theorem~\ref{thm:rigidity1} yields that $\ep^2\Omega\in\big(0,\frac{1}{2 \pi}\big)$. By choosing $\mathcal{R}^+=D$, $\mathcal{R}^-=\C\setminus D$ and $\Gamma=\partial D$, we get that $\tilde{\Psi}\in C^1(\mathcal{R}^+\cup\mathcal{R}^-\cup\Gamma)\cap C^2(\mathcal{R}^+\cup\Gamma)\cap C^2(\mathcal{R}^-\cup\Gamma)$ satisfies the following elliptic system:
	
	\begin{align*}
	F(z,\tilde{\Psi}, D\tilde{\Psi}, D^2\tilde{\Psi}):=\Delta \tilde{\Psi}+2\ep^2\Omega-\frac{1}{\pi}=0, \quad \textnormal{in }\mathcal{R}^+,\\
	G(z,\tilde{\Psi}, D\tilde{\Psi}, D^2\tilde{\Psi}):=\Delta \tilde{\Psi}+2\ep^2\Omega=0, \quad \textnormal{in }\mathcal{R}^-.
	\end{align*}
	In order to use \cite[Theorem 3.1']{kns:freereg}, it remains to prove that
	$$
	\frac{\partial\tilde{\Psi}}{\partial n}\neq 0, \quad\textnormal{on } \Gamma,
	$$
	which follows by the strong maximum principle, the Hopf lemma and Theorem~\ref{thm:rigidity1} since 
	$$
	\Delta \tilde{\Psi}=\frac{1}{\pi}-2\ep^2\Omega>0 \textnormal{ in } D \qquad\text{and}\qquad \tilde{\Psi}=c \quad \textnormal{on } \partial D.
	$$
	Applying \cite[Theorem 3.1']{kns:freereg} provides the analyticity of the boundary of $\Gamma=\partial D$.
\end{proof}

The last ingredient we need for our main result is to show that $\ep$ can only go to $0$ along the local curve of solutions. We begin with a corollary that follows from the rigidity result Theorem~\ref{thm:rigidity1} and Lemma~\ref{thm:rigidity2}.

\begin{corollary}\label{cor:bounds}
	With the notation of Lemma~\ref{thm:rigidity2}, assume in addition that $\phi\in C^{k+\alpha}(\mathbb{T})$, and $\ep_0$ is sufficiently small. Then 
	\begin{equation*}
	\|f\|_{C^{k+\alpha}(\mathbb{T})}\lesssim1,
	\end{equation*}
	where we have $\phi(w)=w+\ep f(w)$, as defined in \eqref{slov5}.
\end{corollary}
\begin{proof}
	Recall from \eqref{nonlinear-op} that we can express the nonlinear operator as
	\begin{align*}
	\mathscr F ^{k+\alpha}(f,\ep)&=\Img\{\mathscr A(f,\ep)(1+\ep f') \}\\&=\Img\{\mathscr A(f,\ep)\ep f'+\mathscr A(f,\ep)\},
	\end{align*}
	where we recall that $\phi(w)=w+\ep f(w)$. We now denote the principal part of this operator by
	\begin{equation*}
	\mathscr F_0^{k+\alpha}(f,\ep)=\Img\{\mathscr A(f,\ep)\ep f'\}.
	\end{equation*}
	We notice that by freezing the coefficient we obtain a Riemann--Hilbert type problem with $\ep\mathscr A$ acting on $f'$. We know from Section~\ref{sect:formulation} that we can rewrite $\ep\mathscr A$ as
	\begin{equation*}
	\ep \mathscr A(f,\ep)=\bigg(\mathscr J(f,\ep)-\Omega(\ep\overline{w+\ep f}+l)\bigg)\ep w+\frac{1}{2\pi}
	\end{equation*}
	with $\mathscr J$ is given in \eqref{J}. From the bounds on $\Omega$ and $f$ in Theorem~\ref{thm:rigidity2}, we see that for $\ep$ sufficiently small, $|\ep\mathscr A|\geq\tfrac{1}{4\pi}$. Classical elliptic theory (see for instance \cite{volpert:book}) now implies that
	\begin{equation*}
	\|f'\|_{C^{k-1+\alpha}(\mathbb{T})}\leq C\big( \|\mathscr F_0^{k+\alpha}(f,\ep)\|_{C^{k-1+\alpha}(\mathbb{T})}+\|f\|_{C^0(\mathbb{T})}\big)
	\end{equation*}
	for some constant $C$ independent of $f$, and hence, by interpolation, we get
	\begin{equation}\label{compactness}
	\|f'\|_{\mathcal{X}^{k-1+\alpha}}\leq C\big( \|\mathscr F^{k+\alpha}(f,\ep)\|_{\mathcal{Y}^{k+\alpha}}+\|f\|_{C^0(\mathbb{T})}\big).
	\end{equation}
	thus concluding the proof.
\end{proof}

\begin{lemma}\label{lemma:bound-ep}
	Assume that \eqref{uniform_bounds_delta} holds for some constant $\delta>0$. Then there exists some $\ep_1(\delta)$ such that $\ep_1(\delta)\leq\ep(s)$ for all $s\in(1,\infty)$ (away from the local curve). In other words,
	\begin{equation*}
		\liminf_{s\to\infty}\ep(s)>0.
	\end{equation*}
\end{lemma}
\begin{proof}
	From Lemma~\ref{thm:rigidity2} and Corollary~\ref{cor:bounds}, we know that there exists an $\ep_0(\delta)$ such that for all $\ep<\ep_0(\delta)$, we have $\|f\|_{C^{k+\alpha}}\lesssim1$ and that $|\Omega|$ is bounded. 
	
	We now argue by contradiction. Suppose that there exists a sequence $s_n\to\infty$ for which we have $\ep(s_n)\to0$. From \eqref{compactness} and compact embeddings of Hölder spaces, we can extract a subsequence so that $\big\{(f(s_n),\ep(s_n)\big\}$ converges in $\mathcal{X}$ to a solution $(f^*,\ep^*)$ of $\mathscr F(f,\ep)=0$ with $\ep^*=0.$ From Lemma~\ref{lem:epsilon0} and continuity of $f$, we must have $f^*=0$. Hence, $\|f(s_n)\|_{\mathcal{X}}\to0$. However, from the uniqueness of solutions in Theorem~\ref{thm:local}\ref{uniqueness} we must have $(f(s_n),\ep(s_n))\in\mathscr C_{\text{loc}}$ for $n$ sufficiently large. This contradicts Theorem~\ref{thm: global bifurcation_intermediate}\ref{thm:global_local}.
\end{proof}

We are now in a position to state our main result.

\begin{theorem}\label{thm:main_abstract}
	There exists a continuous curve $\mathscr C$ of solutions that extends $\mathscr C_{\text{loc}}$, parameterized by $s\in(0,\infty)$, such that
	\begin{enumerate}[label=\rm(\alph*)]
		\item as $s\to\infty$
		\begin{equation*}
		\min\bigg\{\min_{w\in\partial D}\partial_r\tilde{\Psi},\;\min_{\T}|\ep\phi(\tau)+\ep\phi(w)+2l| \bigg\}\to 0;
		\end{equation*}
		\item Provided the bounds \eqref{uniform_bounds_delta} hold for some $\delta>0$, there exists some $\ep_1(\delta)$ such that $\ep(s)\geq\ep_1$ for all $s$ away from the local curve.
		\item For each $s>0,$ the boundary $\partial D$ is analytic.
		\item Near each point $(f(s_0),\ep(s_0))\in\mathscr C$, we can reparameterize $\mathscr C$ so that $s\mapsto (\ep(s),f(s))$ is real analytic.
		\item $(f(s),\ep(s))\notin\mathscr C_{\text{loc}}$ for $s$ sufficiently large, where here $\phi(s)=w+\ep f(s)$.
	\end{enumerate} 
\end{theorem}
\begin{proof}
	The proof of the the theorem follows immediately from combining Theorem~\ref{thm:analytic} and Lemma~\ref{lemma:bound-ep} with Theorem~\ref{thm: global bifurcation_intermediate}.
\end{proof}

\section{Translating pairs}\label{sect:translating}
Till this point, we have focused on the global continuation problem for the rotating pairs \eqref{initial_vorticity}. We now remark that a similar analysis can be done for translating ones. In this case, the weak solutions to \eqref{euler} we seek satisfy the initial data
\begin{equation*}\label{initial_vorticity-t}
\omega_0(z):=\omega(0,z)=\frac{1}{\ep^2\pi}(\chi_{D_1}(z)-\chi_{-D_1}(z)),
\end{equation*}
for some simply-connected bounded domain $D_1$. In particular, the translating solutions we wish to find are of the form
\begin{equation}\label{ansatz-translating}
\omega(t,x)=\omega_0(z-iVt),
\end{equation}
for some constant $V\in \R$. The only possible translating pairs are given by $iV\in \R$ due to the assumed symmetry on $D_1$, such constrain appears in the existence of the local curve \cite{hm:pairs}. Inserting the ansatz \eqref{ansatz-translating} into \eqref{euler} yields the equation
\begin{equation*}\label{translating-eq}
(u_0(z)-iV)\cdot n_{\partial D_1}=0, \quad \textnormal{for all }z\in\partial D_1.
\end{equation*}
As in \eqref{conformal} we choose $D_1=\ep \Phi(\D)+l$, where $\phi:=\Phi\Big|_{\T}$ is such that
$$
\phi(w)=w+\ep f(w), \quad f(w)=\sum_{n\geq 1}a_n w^{-n},
$$
for $a_n\in\R$ and $w\in\T$. In that case, \eqref{ansatz-translating} can be written as
$$
\mathscr G(f,\ep)(w)=0,
$$
where
\begin{equation}\label{nonlinear-op-translating}
\mathscr G(f,\ep):=\Img\bigg\{\bigg(V(f,\ep)+\tfrac{1}{2\pi\ep}\mathcal{C}(w+\ep f)(\overline{w}+\ep\overline{f})-\tfrac{1}{2\pi \ep}\tilde{\mathcal{C}}_{\ep,l}(w+\ep f)(\overline{w}+\ep\overline{f}) \bigg)w(1+\ep f') \bigg\}.
\end{equation}
The speed $V$ is fixed and depends on $(f,\ep)$, analogously as was done for $\Omega$ in \eqref{omega}.

In \cite{hm:pairs}, the authors proved the existence of the local curve of translating pairs of solutions. We recall the result here, paired with the additional result that the linearized operator is invertible along the local curve. 
\begin{theorem}[Local curve of solutions]\label{thm:local_trans}
	There exists $\varepsilon_0>0$ and an analytic function $f(\varepsilon)$ such that $\mathscr G(f(\ep),\ep)=0$, for any $\varepsilon\in[0,\varepsilon_0)$. Moreover, the following properties hold: 
	\begin{enumerate}[label=\rm(\roman*)]
		\item \label{uniqueness_trans} \textup{(Uniqueness)} if $(f,\ep)\in\mathcal{X}$ are sufficiently small, then $\mathscr{G}(f,\ep)=0$ implies $f=f(\ep)$;
		
		\item \label{invertibility_trans} \textup{(Invertibility)} for all $0\leq\ep<\ep_0$, the linearized operator $\mathscr{G}_f(f(\ep),\ep):{\mathcal X}\rightarrow \mathcal Y$ is invertible. 
	\end{enumerate}
\end{theorem}
Proving \ref{invertibility_trans} is simpler here than for corotating pairs since we can rewrite \eqref{nonlinear-op-translating} as 
\begin{align*}
\mathscr G(f,\ep)(w)=&\frac{1}{2\pi}\textnormal{Im}\left[f'(w)\right]+\textnormal{Im}\left[\left\{\mathscr J(f,\ep)(w)+V(f,\ep)\right\}w\phi'(w)\right].\label{G-without-epsilon}
\end{align*}
The rest of the proof then follows exactly as for corotating pairs.

The continuation of the curve now follows completely identically to that for the corotating vortex patches. We remark that for corotating pairs, we needed control on the angular velocity $\Omega$ in order to prove that the boundary of each patch is analytic. Although the rigidity result for corotating pairs does not hold for translating pairs, we also do not need it for the analyticity proof to carry through. We obtain the following result.

\begin{theorem}\label{thm:main-trans}
	There exists a continuous curve $\mathscr C$ of translating vortex patch solutions to \eqref{elliptic formulation}, parameterized by $s\in(0,\infty)$. Moreover, the following properties hold along $\mathscr C$:
	\begin{enumerate}[label=\rm(\roman*)]
		\item \textup{(Bifurcation from point vortex)} The solution at $s=0$ is a pair of points $z_1,z_2$ lying on the horizontal axis at a distance $l$ from each other, translating with constant speed $V_0=\frac{1}{4\pi l}$.
		\item \textup{(Limiting configurations)}\label{thm:alternatives-trans} As $s\to\infty$
		\begin{equation*}\label{thm:physcial min-trans}
		\min\bigg\{\min_{z\in\partial D_1}\ep \nabla\Psi(z)\cdot\bigg(\frac{z-l}{|z-l|}\bigg), \min_{z_m\in\partial D_m}|z_1-z_2| \bigg\}\to0
		\end{equation*}
		\item \textup{(Analyticity)} For each $s>0,$ the boundary $\partial D_m$ is analytic.
		\item \textup{(Graphical boundary)} For each $s > 0$, the boundary of the patch can be expressed as a polar graph.
	\end{enumerate}
\end{theorem}

\appendix
\addtocontents{toc}{\protect\setcounter{tocdepth}{1}}
\section{Bifurcation theorems}\label{app:bif_thms}
First, let us recall the analytic version of the implicit function theorem which is essential in our construction of the local curve of solutions.
\begin{theorem}[Analytic version of Implicit Function theorem \cite{kielhofer}]\label{thm:implicit_function}
	Let $F:U\times V\rightarrow V$ be an analytic function, where $U\subset X$, $V\subset Y$, and where $X,Y,Z$ are real Banach spaces. Let $F(x,y)=0$ have a solution $(x_0,y_0)\in U\times V$ such that
	$$
	\partial_x F(x_0,y_0):X\rightarrow Z,
	$$
	is bounded with a bounded inverse. Then there is a neighborhood $U_1\times V_1$ in $U\times V$ of $(x_0,y_0$ and a mapping $g:V_1\rightarrow U_1\subset X$ such that
	\begin{align*}
	g(y_0)=&x_0,\\
	F(g(y),y)=&0, \textnormal{ for all } y\in V_1.
	\end{align*}
	Moreover, $f$ is analytic on $V_1$. Finally, every solution of $F(x,y)=0$ in $U_1\times V_1$ is of the form $(g(y),y)$.
\end{theorem}
Second, we state the abstract analytic global bifurcation theorem that we will be using. This is an adaptation of the global bifurcation theorem due to Dancer \cite{dancer:global}, later improved by Buffoni and Toland \cite{bt:analytic}. A typical global bifurcation argument would go as follows: one constructs a local curve of solutions by bifurcating from the curve of trivial solutions by means of the Crandall--Rabinowitz theorem. Using a global bifurcation argument one then extends this curve to a global one along which two alternatives can occur: we either get a blow-up scenario or the curve loops back to original bifurcation point. 

Our situation is slightly different here. The formulation of the problem has a singularity at the bifurcation point $(0,0)$. Hence, we cannot continue the curve from that point. Instead, inspired by the global bifurcation theorem developed by Chen, Walsh and Wheeler in \cite{strat} to handle steady solitary water waves, we choose to extend the curve in a suitable open subspace for which the bifurcation parameter at the trivial solution lies on the boundary. This comes at the expense of having an extra alternative that the global curve doesn't connect with the initial bifurcation point as $\ep\to0$. Moreover, since the local curve was not constructed by means of the Crandall--Rabinowitz local bifurcation theorem, we additionally require that the linearized operator be invertible along it. This enables us to glue the local curve together with the global continuation. We remark that as a result of not starting the global curve at the bifurcation point, we cannot have a loop alternative: indeed, the bifurcation point does not lie in the open subspace and the global curve cannot reconnect with the local one due to both uniqueness of solutions, and a lose of analyticity.

\begin{theorem}\label{thm: global bifurcation}
	Let $\mathscr{X}$ and $\mathscr{Y}$ be Banach spaces, $\mathscr{U}$ be an open subset of $\mathscr{X}\times\mathbb{R}$ with $(0,0)\in\partial\mathscr{U}$. Consider a real-analytic mapping $\mathcal{F}\colon\mathscr{U}\to\mathscr{Y}$. Suppose that
	\begin{enumerate}[label=\rm(\Roman*)]
		\item there exists a continuous curve $\mathscr{C}_\textup{loc}$ of solutions to $\mathcal{F}(\mu,x)=0,$ parameterized as 
		\begin{equation*}
		\mathscr{C}_{\textup{loc}}:=\{(\mu,\tilde{x}(\mu)):0<\mu<\mu_* \}\subset\mathcal{F}^{-1}(0),
		\end{equation*}
		for some $\mu_*>0$ and continuous $\tilde{x}$ with values in $\mathscr{X}$ and $\lim_{\mu\searrow0}\tilde{x}(\mu)=0$;
		\item \label{invertible} the linearized operator $\mathcal{F}_x(\mu,\tilde{x}(\mu))\colon\mathscr{X}\to\mathscr{Y}$ is invertible for all $\mu$;
		\item \label{semifredholm} for any $(\mu,x)\in\mathscr{U}$ with $\mathcal{F}(\mu,x)=0$ the Fréchet derivative $\mathcal{F}_x(\mu,x)\colon\mathscr{X}\to\mathscr{Y}$ is Fredholm with index $0$;
		\item for some sequence $(Q_j)_{j=1}^\infty$ of bounded closed subsets of $\mathscr U$ with $\mathscr U=\cup_j Q_j$, the set $\{(\mu,x)\in\mathscr U:\mathcal{F}(\mu,x)=0\}\cap Q_j$ is compact for each $j$. 
	\end{enumerate}
	Then $\mathscr{C}_\text{loc}$ is contained is a curve of solutions $\mathscr{C}$,  parameterized as
	\begin{equation*}
	\mathscr{C}:=\{(\mu(s),x(s)):0<s<\infty \}\subset\mathcal{F}^{-1}(0)
	\end{equation*}
	for some continuous $(0,\infty)\ni s\mapsto(x(s),\mu(s) )\in\mathscr{U},$ with the following properties.
	\begin{enumerate}[label=\rm(\alph*)]
		\item for every $j\in\mathbb{N}$ there exists $s_j>0$ such that $(\mu(s),x(s))\notin Q_j$ for $s>s_j$.
		
		\item  Near each point $(\mu(s_0),x(s_0))\in\mathscr{C}$, we can reparameterize $\mathscr{C}$ so that $s\mapsto(\mu(s),x(s))$ is real analytic.
		
		\item $(\mu(s),x(s))\notin\mathscr{C}_{\text{loc}}$ for $s$ sufficiently large.
	\end{enumerate}
\end{theorem}
\begin{proof}
	The proof is almost completely identical to the global bifurcation theorem of Buffoni--Toland (see \cite{bt:analytic}). The only part we need to deal with the fact that the bifurcation point does not lie in the subspace $\mathscr U$ we are working in. To this end, we follow the ideas used in \cite[Theorem 6.1]{strat}. We provide a brief sketch. As in \cite{strat}, since \ref{invertible} holds we can construct the distinguished arc $A_0$, the connected component of
	\begin{equation*}
	\mathcal{A}:=\big\{(\mu,x)\in\mathscr{U}:\mathcal{F}(\mu,x)=0,\; \mathcal{F}_x(\mu,x)\text{ is invertible }  \big\}
	\end{equation*}
	in which $(\mu_{1/2},x_{1/2}):=(\mu_*/2,x(\mu_*/2))$ lies. The analytic implicit function theorem guarantees that all distinguished arcs are graphs. After possibly re-parameterizing, we write $A_0$ as
	\begin{equation*}
	A_0=\{(\mu(s),x(s)):0<s<1 \},
	\end{equation*}
	where $\mu(s)$ is increasing. From the implicit function theorem, the local curve of solutions $\mathscr{C}_\text{loc}$ lies entirely in $A_0$. Arguing as in the proof of \cite[Theorem 6.1]{strat}, the starting point of $A_0$ is
	\begin{equation*}
	\lim_{s\searrow0}(\mu(s),x(s))=(0,0).
	\end{equation*}
	The next step is to consider the limit $s\nearrow1$. This now follows exactly as in \cite{bt:analytic}.
\end{proof}
\section{Cited results}
\subsection{Results for Riemann--Hilbert problems}
We will often treat the problem as a Riemann--Hilbert problem. The following result will be useful. 
\begin{lemma}[Linear Riemann--Hilbert problems]\label{lem:RH_problem}
	Suppose $a\in C^{k-1+\alpha}(\mathbb{T},\mathbb{C})$ has winding number $0$ in that
	\begin{equation*}
	|a|>0\qquad\text{and}\qquad\operatorname{arg}a(e^{it})\bigg|_{t=0}^{t=2\pi}=0
	\end{equation*}
	and also that $a$ has the symmetry property $a(\overline{w})=\overline{a(w)}$. Then:
	\begin{enumerate}[label=\rm(\alph*)]
		\item The problem
		\begin{equation*}
		\Img\{ag'\}=0\quad\text{on }\mathbb{T},\qquad (g-w)/\ep\in\mathcal{X}^{k+\alpha},
		\end{equation*}
		has a unique solution $g=g_0,$ whose derivative is given explicitly by
		\begin{equation*}
		g_0'(w)=\operatorname{exp}\bigg\{\frac{w}{2\pi}\int_{\T}\frac{\tau^{-1}\theta(\tau)-w^{-1}\theta(w)}{\tau-w}d\tau \bigg\},
		\end{equation*}
		where here
		\begin{equation*}
		\theta(w)=\operatorname{arg}\frac{a(w)}{\overline{a(w)}},
		\end{equation*}
		and the branch of the $\operatorname{arg}$ function is fixed by requiring $\theta(1)=0$.
		\item The operator
		\begin{equation*}
		L\colon \mathcal{X}^{k+\alpha}\to\mathcal{Y}^{k-1+\alpha},\qquad g\mapsto\Img\{ag'\}
		\end{equation*}
		is well-defined and invertible, with inverse operator characterized by
		\begin{equation*}
		\frac{d}{dw}L^{-1}h(w)=-\frac{wg_0'(w)}{\pi}\int_{\T}\frac{1}{\tau-w}\bigg(\frac{h(\tau)}{a(\tau)g_0'(\tau)\tau}-\frac{h(w)}{a(w)g_0'(w)w} \bigg)d\tau.
		\end{equation*}
	\end{enumerate}
\end{lemma}

In order to use the above lemma, we need to ensure that the winding number of $\ep A$ is indeed $0$.

\begin{lemma}[Winding number $0$]\label{lem:winding_number}
	Suppose that $(\phi,\ep)$ is a solution to \eqref{A_formulation} and that \eqref{uniform_bounds_delta} holds for some $\delta>0$. Then $\ep A$, where $A$ is defined as in \eqref{A} has winding number $0$ in the sense that
	\begin{equation}\label{winding number 0}
	\ep A\neq0,\qquad\operatorname{arg}\ep A(e^{it})\bigg|_{t=0}^{t=2\pi}=0,
	\end{equation}
	for some continuous branch of $\operatorname{arg}$.
\end{lemma}
\begin{proof}
	From \eqref{A_formulation} we know that $\Img(A\phi')=0$ and from \eqref{uniform_bounds_delta}, we know that $\ep A\neq0$ and $\phi'\neq0$. This ensures that $\ep A$ can be expressed in terms of $\lambda\phi'$ where $\lambda$ is some real-valued, non-vanishing function in $C^{k-1+\alpha}(\mathbb{T})$. As a result, it suffices to show that $\phi'$ has winding number $0$ in the sense of \eqref{winding number 0}. Recall the holomorphic extension $\Phi$ of $\phi$ to $\mathbb{C}\setminus\mathbb{D}$. Then $f\in C^{k+\alpha}$, defined as in \eqref{slov5}, implies that 
	\begin{equation*}
	\frac{w\Phi''}{\Phi'}\to0\qquad\text{as }|w|\to\infty. 
	\end{equation*}	
	Therefore, we have
	\begin{equation*}
	\operatorname{arg}\phi'(e^{it})\bigg|_{t=0}^{t=2\pi}=\lim_{r\to\infty}\operatorname{arg}\Phi'(re^{it})\bigg|_{t=0}^{t=2\pi}=\lim_{r\to\infty}\frac{1}{2\pi i}\int_{|w|=r}\frac{\Phi''(w)}{\Phi'(w)}dw=0,
	\end{equation*}
	thus concluding the proof.
\end{proof}

\subsection{Uniform bounds}\label{app:uniform_bounds}
In order to prove control of $\|\phi\|_{C^{1+\alpha}}$, we will use the following results. Since they are very similar to the ones in \cite{hmw:global}, we only state them and provide outlines of the proofs for the sake of completeness.

\begin{lemma}\label{lem:Koebe_1/4}
	Any $(\phi,\ep)$ solution to \eqref{A_formulation} satisfies $\|\phi\|_{L^\infty}\leq4$.
\end{lemma}
\begin{proof}
	The proof of this lemma follows directly from applying the Koebe $1/4$ lemma to the function $g(\zeta):=1/\Phi(1/\zeta)$, where $\Phi$ is the extension to a conformal mapping $\mathbb{C}\setminus\overline{\mathbb{D}}\to\mathbb{C}$ of $\phi$.
\end{proof}
The next lemma tells us that a bound
\begin{equation*}
\bigg\|\operatorname{arg}\frac{w\phi'(w)}{\phi(w)}\bigg\|_{L^\infty}<\frac{\pi}{2}
\end{equation*}
implies a bound on $\|\phi'\|_{L^p}$, for some $p>1$. 
\begin{lemma}\cite{gaier:conformal}\label{lem:bounds_on_phi'}
	Suppose $(\phi,\ep)$ solves \eqref{A_formulation}. Then we have
	\begin{equation*}
	\int_{0}^{2\pi}|\phi'(e^{it})|^p\,dt\leq\frac{2\pi\cdot4^p}{\cos(p\|\gamma\|_{L^\infty})}\qquad\text{for }0\leq p<\frac{\pi/2}{\|\gamma\|_{L^\infty}},
	\end{equation*}
	where here
	\begin{equation*}
	\gamma(w)=\operatorname{arg}\frac{w\phi'(w)}{\phi(w)}.
	\end{equation*}
\end{lemma}
\begin{proof}
	As for the proof of the previous lemma, we use $\Phi$, the extention to a conformal mapping of $\phi$. Specifically, we define the holomorphic function
	\begin{equation*}
	F(w):=\log\frac{w\Phi'(w)}{\Phi(w)}=u(w)+i\gamma(w).
	\end{equation*}
	Using the calculus of residues we get
	\begin{equation*}
	1=\frac{1}{2\pi}\int_{0}^{2\pi}e^{pu(e^{it})}\cos[p\gamma(e^{it})]dt,
	\end{equation*}
	for any $p$. By adding the restriction $p\|\gamma\|_{L^\infty}<\pi/2$, we obtain
	\begin{equation*}
	\int_{0}^{2\pi}\frac{|\Phi'(e^{it})|^p}{|\Phi(e^{it})|^p}dt\leq\frac{2\pi}{\cos(p\|\gamma\|_{L^\infty})}.
	\end{equation*}
	We can now apply Lemma~\ref{lem:Koebe_1/4} to get the desired result.
\end{proof}

\section*{Statements and declarations}

The authors certify that they have no affiliations with or involvement in any organization or entity with any financial interest, or non-financial interest in the subject matter or materials discussed in this manuscript. Moreover, data sharing not applicable to this article as no datasets were generated or analyzed during the current study.

\bibliographystyle{amsalpha}
\bibliography{references}

\end{document}